\documentclass[12pt]{amsart}

\usepackage{amsmath}
\usepackage{amssymb}
\usepackage{pdfsync}

\newcommand{\C}{{\mathbb C}}       
\newcommand{\R}{{\mathbb R}}       
\newcommand{\Z}{{\mathbb Z}}       
\newcommand{\DD}{{\mathcal D}}

\newcommand{\HH}{{\mathcal H}}

\newcommand{\WW}{{\mathcal W}}

\newcommand{\CC}{{\mathcal C}}

\newcommand{\diam}{{\rm diam}}
\newcommand{\dist}{{\rm dist}}

\newcommand{\fiproof}{{\hspace*{\fill} $\square$ \vspace{2pt}}}

\newcommand{\imag}{{\rm Im}}

\newcommand{\rf}[1]{{(\ref{#1})}}

\newcommand{\supp}{{\rm supp}}

\newcommand{\vphi}{{\varphi}}
\newcommand{\ve}{{\varepsilon}}
\newcommand{\vv}{{\vspace{2mm}}}
\newcommand{\vvv}{{\vspace{3mm}}}
\newcommand{\wt}[1]{{\widetilde{#1}}}
\newcommand{\wh}[1]{{\widehat{#1}}}

\newcommand{\rest}{{\lfloor}}

\newtheorem{theorem}{Theorem}[section]
\newtheorem{lemma}[theorem]{Lemma}

\newtheorem{coro}[theorem]{Corollary}

\newtheorem*{theorem*}{Theorem}

\theoremstyle{definition}

\theoremstyle{remark}
\newtheorem{rem}[theorem]{Remark}

\numberwithin{equation}{section}

\newcommand{\brem}{\begin{rem}}
\newcommand{\erem}{\end{rem}}

\begin{document}

\title[The Beurling transform in Lipschitz domains]{Smoothness of the Beurling transform in Lipschitz domains}

\author{Victor Cruz} 
\address{Victor Cruz. Universidad Tec\-no\-l\'o\-gi\-ca de la Mix\-te\-ca, M\'exico}\email{victorcruz@mixteco.utm.mx}

\author{Xavier Tolsa}

\address{Xavier Tolsa. Instituci\'{o} Catalana de Recerca
i Estudis Avan\c{c}ats (ICREA) and Departament de
Ma\-te\-m\`a\-ti\-ques, Universitat Aut\`onoma de Bar\-ce\-lo\-na,
Catalonia} \email{xtolsa@mat.uab.cat}

\thanks{V.C.\ was supported partially by grants 2009SGR-000420 (Catalonia), MTM-2010-15657 (Spain), 
and NF-129254 (Spain).
X.T.\ was supported partially by grants
2009SGR-000420 (Catalonia) and MTM-2010-16232 (Spain).}

\maketitle

\begin{abstract}
Let $\Omega\subset\C$ be a Lipschitz 
domain and consider the Beurling transform of $\chi_\Omega$:
$$B\chi_\Omega(z)=\lim_{\ve\to0}\frac{-1}\pi\int_{w\in\Omega,|z-w|>\ve}\frac1{(z-w)^2}\,dm(w).$$
Let $1<p<\infty$ and $0<\alpha<1$ with $\alpha p>1$.
In this paper we show that if the outward unit normal $N$ on $\partial\Omega$ belongs to
the Besov space ${B}_{p,p}^{\alpha-1/p}(\partial\Omega)$, then $B\chi_\Omega$
is in the Sobolev space $W^{\alpha,p}(\Omega)$. This result is sharp. Further, 
together with recent results by Cruz, Mateu and Orobitg, this implies that
 the Beurling transform is bounded
in $W^{\alpha,p}(\Omega)$ if 
$N$  
belongs to ${B}_{p,p}^{\alpha-1/p}(\partial\Omega)$, assuming that 
$\alpha p>2$.
\end{abstract}


\section{Introduction}

In this paper we obtain sharp results on the Sobolev regularity of the Beurling transform
of the characteristic function of Lipschitz domains. 
It has been shown recently in \cite{CMO} that this plays a crucial
role in the boundedness of the Beurling transform in the Sobolev spaces on domains.

Recall that the Beurling transform of a locally integrable function $f:\C\to\C$ is defined by the
following singular integral:
$$Bf(z)=\lim_{\ve\to0}\frac{-1}\pi\int_{|z-w|>\ve}\frac{f(x)}{(z-w)^2}\,dm(w)\qquad z\in\C,$$
whenever the limit and the integral makes sense.
It is well known that for $f\in L^p(\C)$, for some $1\leq p <\infty$, the limit above exists a.e.

The Beurling transform is an operator of great importance for the study of quasiconformal mappings
in the plane, due to the fact that it intertwines the $\partial$ and
$\bar \partial$ derivatives. Indeed, in the sense of distributions, one has
$$B(\bar\partial f)=\partial f.$$

Let $\Omega\subset\C$ be a bounded domain (open and connected). We say that $\Omega\subset\C$ is a $(\delta,R)$-Lipschitz domain if for each $z\in\partial\Omega$
there exists a Lipschitz function $A:\R\to\R$ with slope $\|A'\|_\infty\leq\delta$ such that, after
a suitable rotation,
$$\Omega\cap B(z,R) = \bigl\{(x,y)\in B(z,R):\,y>A(x)\bigr\}.$$
If we do not care about the constants $\delta$ and $R$, then we just say that $\Omega$ is a
Lipschitz domain. 

Also, we call an open set $\Omega$  a special $\delta$-Lipschitz domain if the exists
a Lipschitz function $A:\R\to\R$ with compact support such that
$$\Omega=\{(x,y)\in\C:\,y>A(x)\}.$$
As above, if we do not care about $\delta$, then we just say that $\Omega$ is a special
Lipschitz domain.

If in the definitions of Lipschitz and special Lipschitz domains, moreover, one asks $A$ to be of class $\CC^1$,
then  $\Omega$ is called a $\CC^1$ or a special $\CC^1$ domain, respectively.

The results that we obtain in this paper deal with the Sobolev smoothness of order $0<\alpha\leq1$ of $B\chi_
\Omega$ on $\Omega$, which depends on the Besov regularity of the boundary $\partial \Omega$. 
For the precise definitions of the Sobolev spaces $W^{\alpha,p}$ and the Besov spaces 
$B_{p,q}^\alpha$, see Section \ref{sec2}. Our first theorem concerns the Sobolev spaces $W^{1,p}(\Omega)$:

\begin{theorem}\label{teo}
Let $\Omega\subset\C$ be a either $(\delta,R)$-Lipschitz domain or a special $\delta$-Lipschitz domain, and let $1< p<\infty$.
Denote by $N(z)$ the outward unit normal of $\Omega$ in $z\in\partial\Omega$.
If $N\in \dot B_{p,p}^{1-1/p}(\partial \Omega)$, then $B(\chi_\Omega) \in \dot W^{1,p}(\Omega)$.
Moreover,
$$\|\partial B(\chi_\Omega)\|_{L^p(\Omega)}\leq c
\|N\|_{\dot {B}_{p,p}^{1-1/p}(\partial\Omega)},$$
with $c$ depending on $p$, $\delta$ and, in case $\Omega$ is a Lipschitz domain, on $R$.
\end{theorem}

\vvv

Above,
 $\dot W^{1,p}(\Omega)$ stands for the homogeneous Sobolev space on $\Omega$ consisting of the functions
whose distributional derivatives belong to $L^p(\Omega)$, while 
$\dot {B}_{p,p}^{1-1/p}(\partial \Omega)$ is the homogeneous Besov
space on $\partial\Omega$ associated to the indices $p,p$, with regularity $1-1/p$. 
See Section \ref{sec2} for more details. 

Also, let us remark that, as $B(\chi_\Omega)$ is analytic in $\Omega$,
it turns out that 
$$\|\partial B(\chi_\Omega)\|_{L^p(\Omega)}\approx \| B(\chi_\Omega)\|_{\dot W^{1,p}(\Omega)}.$$

For the fractional Sobolev spaces $W^{\alpha,p}(\Omega)$ for $0<\alpha<1$, we will prove the following result,
which is analogous to Theorem \ref{teo}:

\begin{theorem}\label{teo2}
Let $\Omega\subset\C$ be either a $(\delta,R)$-Lipschitz or 
a special $\delta$-Lipschitz domain, 
and let $1< p<\infty$ and $0<\alpha<1$ such that $\alpha\,p>1$.
Denote by $N(z)$ the outward unit normal of $\Omega$ in $z\in\partial\Omega$.
If $N\in B_{p,p}^{\alpha-1/p}(\partial \Omega)$, then $B(\chi_\Omega) \in \dot W^{\alpha,p}(\Omega)
\cap \dot B_{p,p}^\alpha(\Omega)$.
Moreover,
$$\|B(\chi_\Omega)\|_{\dot W^{\alpha,p}(\Omega)} + \|B(\chi_\Omega)\|_{\dot B^\alpha_{p,p}(\Omega)}\leq c\,
\|N\|_{\dot {B}_{p,p}^{\alpha-1/p}(\partial\Omega)},$$
with $c$ depending on $p$, $\alpha$, $\delta$ and, in case $\Omega$ is a Lipschitz domain, on $R$. 
\end{theorem}

Recall that the Beurling transform is bounded in $L^p(\C)$. Thus, saying that $B(\chi_\Omega) \in \dot W^{\alpha,p}(\Omega)$ is equivalent to saying that $B(\chi_\Omega) \in W^{\alpha,p}(\Omega)$
if $\Omega$ is bounded. Analogously, in the same situation, $B(\chi_\Omega)\in \dot B_{p,p}^\alpha(\Omega)$ if 
and only if $B(\chi_\Omega)\in B_{p,p}^\alpha(\Omega)$.

The Besov spaces $B_{p,p}^{\alpha-1/p}$ appear naturally in the context of Sobolev spaces.
Indeed, it is well known that the traces of the functions from $W^{\alpha,p}(\Omega)$ on $\partial\Omega$ coincide with the functions from 
$B_{p,p}^{\alpha-1/p}(\partial\Omega)$, whenever $\Omega$ is a Lipschitz domain. So, by combining this fact 
with Theorems \ref{teo} and \ref{teo2}, 
 one deduces that $B(\chi_\Omega)\in W^{\alpha,p}(\Omega)$
if $N$ is the trace of some (vectorial) function from $W^{1,p}(\Omega)$.

The results stated in Theorems \ref{teo} and \ref{teo2} are sharp. In fact, it has been proved in 
\cite{Tolsa-beurling}, for $0<\alpha\leq1$ with $\alpha\,p>1$, that if $\Omega$ is a $\CC^1$ domain and 
$B(\chi_\Omega) \in \dot W^{\alpha,p}(\Omega)$, then $N\in\dot {B}_{p,p}^{\alpha-1/p}(\partial\Omega)$.
So one deduces that
$$B(\chi_\Omega)\in \dot W^{\alpha,p}(\Omega)\quad\Longleftrightarrow \quad N\in \dot{B}_{p,p}^{\alpha-1/p}(\partial\Omega),\quad \mbox{for $0<\alpha\leq1$ with $\alpha\,p>1$,}$$
assuming $\Omega$ to be a $\CC^1$ domain. This shows that the smoothness of $B(\chi_\Omega)$ characterizes the
Besov regularity of the boundary $\partial\Omega$.

The hypotehsis $\alpha p> 1$ for our results is quite natural. We will prove below (see
Section \ref{secfi}) that if $\alpha\,p<1$, then $B(\chi_\Omega)\in \dot W^{\alpha,p}(\Omega)$ (and
in the case $\alpha<1$, $B(\chi_\Omega)\in\dot B^\alpha_{p,p}(\Omega)$ too), without any assumption on the 
Besov regularity of the boundary. In the endpoint case $\alpha p=1$ we will also obtain other partial
results (see Section \ref{secfi} again).

Our motivation to understand when $B\chi_\Omega\in W^{1,p}
(\Omega)$ arises from the results of Cruz, Mateu and Orobitg in \cite{CMO}. 
In this paper one studies the smoothness of quasiconformal mappings when the Beltrami coefficient belongs to $W^{\alpha,p}(\Omega)$, for some fixed $1<p<\infty$ and $0<\alpha<1$.
An important step in the arguments is the following kind of $T1$ theorem:

\begin{theorem*}[\cite{CMO}]
Let $\Omega\subset\C$ be a bounded $\CC^{1+\ve}$ domain, for some $\ve>0$, and let $1<p<\infty$ and $0<\alpha\leq1$ be such
that $\alpha\,p>2$. Then, the Beurling transform is bounded
in $W^{\alpha,p}(\Omega)$ if and only if $B(\chi_\Omega)\in W^{\alpha,p}(\Omega)$.
\end{theorem*}

As a corollary of the preceding theorem and the results of this paper one obtains the following.

\begin{coro}
Let $\Omega\subset\C$ be a bounded Lipschitz domain and $0<\alpha\leq1$, $1<p<\infty$ such
that $\alpha\,p>2$. If the outward unit normal of $\Omega$ is in the Besov space $\dot B_{p,p}^{\alpha-1/p}(\partial\Omega)$, then the Beurling transform is bounded
in $W^{\alpha,p}(\Omega)$.
\end{coro}

Let us remark that, by Lemma \ref{lemanorm} below, the fact that $N\in \dot B_{p,p}^{\alpha-1/p}(\partial\Omega)$ implies that the local parameterizations of the boundary can be taken from $B_{p,p}^{1+\alpha-1/p}(\R)\subset \CC^{1+\ve}(\R)$ because $\alpha p>2$, and thus theorem from Cruz-Mateu-Orobitg applies.

Let us also mention that the boundedness of the Beurling transform in the Lipschitz spaces 
$\rm Lip_\ve(\Omega)$ for
domains $\Omega$ of class $\CC^{1+\ve}$ has been studied previously in \cite{Mateu-Orobitg-Verdera}, \cite{Li-Vogelius},
and \cite{Depauw}, because of the applications to quasiconformal mappings and PDE's.

It is well known that the Beurling transform of the characteristic function of a ball
vanishes identically inside the ball. Analogously, the Beurling transform of the characteristic
function of a half plane  is constant in the half plane, and also in the complementary half
plane. So its derivative vanishes everywhere except in its boundary. This fact will play a crucial
role in the proofs of Theorems \ref{teo} and \ref{teo2}. Roughly speaking, the arguments consist in 
comparing $B(\chi_\Omega)(x)$ (or an appropriate ``$\alpha$-th derivative'') to $B(\chi_\Pi)(x)$
(or to the analogous ``$\alpha$-th derivative''), where $\Pi$ is some half plane that approximates
$\Omega$ near $x\in\Omega$.
The errors are estimated in terms of the so called $\beta_1$ coefficients.
Given an interval $I\subset\R$ and a function $f\in L^1_{loc}$, one sets
\begin{equation}\label{defbeta} 
\beta_1(f,I) = \inf_\rho \frac1{\ell(I)}\int_{3I} \frac{|f(x)-\rho(x)|}{\ell(I)}\,dx,
\end{equation}
where the infimum is taken over all the affine functions $\rho:\R\to\R$.
The coefficients $\beta_1$'s (and other variants $\beta_p$, 
$\beta_\infty$,\ldots) appeared first in the works of Jones \cite{Jones-traveling} and David and Semmes
\cite{DS1} on quantitative rectifiability. 
They have become a useful tool in problems which involve geometric measure theory and multi-scale analysis. See \cite{DS2}, \cite{Leger}, \cite{Mas-Tolsa}, \cite{Tolsa-bilip}, or \cite{Tolsa-jfa}, for example, besides the aforementioned references. Finally, the connection with the Besov
smoothness from the boundary arises from a nice
 characterization of Besov spaces in terms of $\beta_1$'s due to
Dorronsoro \cite{Dorronsoro2}.

The plan of the paper is the following. In Section \ref{sec2}, some preliminary notation and background 
is reviewed. 
In Section \ref{sec3} we prove some auxiliary lemmas which will be used later.
In Section \ref{sec4} we obtain more auxiliary results necessary for Theorem
\ref{teo}, which is proved in the subsequent section. 
Sections \ref{sec6}, \ref{sec7} and \ref{sec8} are devoted to Theorem \ref{teo2}. 
The final Section \ref{secfi} contains some results for the case $\alpha p\leq1$.


\vvv
\section{Preliminaries}\label{sec2}

As usual, in the paper the letter `$c$' stands for some
constant (quite often an absolute constant) which may change its value at different occurrences. On
the other hand, constants with subscripts, such as $c_0$, retain
their values at different occurrences. The notation $A\lesssim B$
means that there is a fixed positive constant $c$ such that
$A\leq cB$. So $A\approx B$ is equivalent to $A\lesssim B \lesssim
A$. 

For $n\geq 2$ we will denote the Lebesgue measure in $\R^n$  by $m$ or $dm$. On the other hand, 
for $n=1$, we will use the typical notation $dx$, $dy$,\ldots 



\subsection{Dyadic and Whitney cubes} \label{sec2.1}

By a cube in $\R^n$ (in our case $n=1$ or $2$) we mean a cube with edges parallel to the axes. Most of the cubes in our paper will be
dyadic cubes, which are assumed to be half open-closed. The collection of all dyadic cubes is denoted by $\DD(\R^n)$. They are called intervals for $n=1$ and squares for $n=2$.
The side length of a cube $Q$ is written as $\ell(Q)$, and its center as $z_Q$. 
The lattice of dyadic cubes of side length $2^{-j}$ is
denoted by $\DD_j(\R^n)$. 
Also, given $a>0$ and any cube $Q$, we denote by $a\,Q$ the cube concentric with $Q$ with side length $a\,
\ell(Q)$.

Recall that any open subset $\Omega\subset\R^n$ can be decomposed in the so called Whitney cubes, as follows:
$$\Omega = \bigcup_{k=1}^{\infty} Q_k, $$
where $Q_k$ are disjoint dyadic cubes (the ``Whitney cubes'') such
that for some constants $\rho>20$ and $D_0\geq1$ the following holds,
\begin{itemize}
\item[(i)] $5Q_k \subset \Omega$.
\item[(ii)] $\rho Q_k \cap \Omega^{c} \neq \varnothing$.
\item[(iii)] For each cube $Q_k$, there are at most $D_0$ squares $Q_j$
such that $5Q_k \cap 5Q_j \neq \varnothing$. Moreover, for such squares $Q_k$, $Q_j$, we have 
$\frac12\ell(Q_k)\leq
\ell(Q_j)\leq 2\,\ell(Q_k)$.
\end{itemize}
We will denote by $\WW(\Omega)$ the family $\{Q_k\}_k$ of Whitney cubes of $\Omega$.

If $\Omega\subset\C$ is a Lipschitz domain, then $\partial\Omega$ is a chord arc curve.
Recall that a chord arc curve is just the bilipschitz image of 
a circumference. Then one can define a family $\DD(\partial\Omega)$ of ``dyadic'' arcs  which play the same role
as the dyadic intervals in $\R$: for each $j\in\Z$ such that $2^{-j}\leq \HH^1(\partial\Omega)$, $\DD_j(\partial\Omega)$ is a partition of
$\partial \Omega$ into pairwise disjoint arcs of length $\approx2^{-j}$, and 
$\DD(\partial\Omega)=\bigcup_j \DD_j(\partial\Omega)$. As in the case of $\DD(\R^n)$, two
arcs from $\DD(\partial\Omega)$ either are disjoint or one contains the other. 

If $\Omega$ is a special Lipschitz domain, that is, 
$\Omega=\{(x,y)\in\C:\,y>A(x)\}$, where $A:\R\to\R$ is a Lipschitz function, there exists
an analogous family $\DD(\partial
\Omega)$. In this case, setting $T(x)=(x,A(x))$, one can take $\DD(\partial\Omega)=T(\DD(\R))$, 
for instance.

If $\Omega$ is either a Lipschitz or a special Lipschitz domain, to each
$Q\in\WW(\Omega)$ we assign a cube $\phi(Q)\in\DD(\partial \Omega)$ such that $\phi(Q)\cap\rho Q\neq
\varnothing$ and $\diam(\phi(Q))\approx\ell(Q)$.
So there exists  some big constant $M$ depending on the parameters of the Whitney decomposition and on the chord arc constant of $\partial \Omega$ such that
$$\phi(Q)\subset M\,Q,\qquad \text{and}\qquad Q\subset B(z,M\ell(\phi(Q)))\quad
\mbox{for all $z\in\phi(Q)$}.$$
From this fact, it easily follows that there exists some constant $c_2$ such that for every $Q\in
\WW(\Omega)$,
$$\#\{P\in\DD(\partial\Omega):\,P=\phi(Q)\}\leq c_2.$$

Conversely, to each
$Q\in\DD(\partial\Omega)$ we assign a square $\psi(Q)\in\WW(\Omega)$ such that $\diam(\psi(Q))
\approx \dist(Q,\psi(Q))\approx \ell(Q)$. One may think of $\psi$ as a kind of inverse of $\phi$.
As above,
 there exists some constant $c_3$ such that for every $Q\in
\DD(\partial\Omega)$,
$$\#\{P\in\WW(\partial\Omega):\,P=\psi(Q)\}\leq c_3.$$


\subsection{Sobolev spaces} \label{subsec2.2}

Recall that for an open domain $\Omega\subset \R^n$, $1\leq p<\infty$, and a positive integer $m$, the Sobolev space $W^{m,p}(\Omega)$
consists of the functions $f\in L^1_{loc}(\Omega)$ such that
$$\|f\|_{W^{m,p}(\Omega)} = \biggl(\,\sum_{0\leq|\alpha|\leq m} \|D^\alpha f\|_{L^p(\Omega)}^p\biggr)^{1/p}
<\infty,$$
where $D^\alpha f$ is the $\alpha$-th derivative of $f$, in the sense of distributions. 
The homogeneous Sobolev seminorm $\dot W^{m,p}$ is defined by
$$\|f\|_{\dot W^{m,p}(\Omega)} :=   \biggl(\,\sum_{|\alpha|=m} \|D^\alpha f\|_{L^p(\Omega)}^p\biggr)^{1/p}.$$

For a non
integer $0<\alpha<1$, one sets
$$D^{\alpha}f(x)=\left(\int_{\Omega}\frac{|f(x)-f(y)|^{2}}{|x-y|^{n+2\alpha}}\,dm(y)\right)^{\frac{1}{2}},$$
and then
$$\|f\|_{W^{\alpha,p}(\Omega)} = \biggl(\|f\|_{L^p(\Omega)}^p + \|D^\alpha f\|_{L^p(\Omega)}^p \biggr)^{1/p}.$$
See \cite{Strich}, for example.
The homogeneous Sobolev seminorm $\dot W^{\alpha,p}(\Omega)$ equals
$$\|f\|_{\dot W^{\alpha,p}(\Omega)} =  \|D^\alpha f\|_{L^p(\Omega)}.$$


\subsection{Besov spaces} \label{subsec2.3}
In this section we review some basic results concerning Besov spaces. We pay special attention to the homogeneous Besov spaces $\dot B_{p,p}^\alpha$, with $0<\alpha<1$.

Consider a radial $\CC^\infty$ function $\eta:\R^n\to\R^n$ whose Fourier transform $\wh \eta$ is supported in the annulus
$A(0,1/2,3/2)$, such that setting $\eta_{k}(x)=\eta_{2^{-k}}(x) = 2^k\,\eta(2^k\,x)$,
\begin{equation}\label{eqnu3}
\sum_{k\in\Z} \wh{\eta_{k}}(\xi)=1\qquad \mbox{for all $\xi\neq0$.}
\end{equation}
Then, for $f\in L^1_{loc}(\R^n)$, $1\leq p,q<\infty$, and $\alpha>0$, one defines the seminorm 
$$\|f\|_{\dot{B}_{p,q}^{\alpha}}= \left( \sum_{k\in\Z}\|2^{k\alpha}\eta_{k}*f\|_p^q\right)^{1/q},$$
and the norm 
$$\|f\|_{{B}_{p,q}^{\alpha}}=\|f\|_p +\|f\|_{\dot{B}_{p,q}^\alpha}.$$
The homogeneous Besov space $\dot{B}_{p,q}^{\alpha}\equiv \dot{B}_{p,q}^{\alpha}(\R^n)$ consists of the functions such that $\|f\|_{\dot{B}_{p,q}^{\alpha}}<\infty$, while the functions in the Besov space ${B}_{p,q}^{\alpha}\equiv {B}_{p,q}^{\alpha}(\R^n)$ 
are those such that $\|f\|_{{B}_{p,q}^{\alpha}}<\infty$.
If one chooses a function different from $\eta$ which satisfies the same properties as $\eta$
above, then one obtains an equivalent seminorm and norm, respectively.

Given $f\in L^1_{loc}(\R^n)$ and $h>0$, denote $\Delta_h(f)(x)=f(x+h)-f(x)$.
For $1\leq p,q<\infty$ and $0<\alpha<1$, it turns out that
\begin{equation}\label{eqgg0}
\|f\|_{\dot{B}_{p,q}^\alpha}^p \approx
\int_{\R^n} \frac{\|\Delta_h(f)\|_q^p}{|h|^{\alpha p + n}}\,dm(h),
\end{equation}
assuming $f$ to be compactly supported, say. Otherwise the comparability is true modulo polynomials, that
is, above one should replace $\|\Delta_h(f)\|_q$ by 
$$\inf_{p \text{ polynomial}} \|\Delta_h(f+p)\|_q.$$
See \cite[p.\ 242]{Trieble}, for instance. Analogous characterizations hold for Besov spaces with regularity $\alpha\geq1$. 
In this case it is necessary to use differences of higher order.

Observe that, for $p=q$ and $0<\alpha<1$, one has
\begin{equation}\label{eqbpp1}
\|f\|_{\dot{B}_{p,p}^\alpha}^p \approx
\iint_{\R^n\times\R^n} \frac{|\Delta_h(f)|^p}{|h|^{\alpha p + n}}\,dm(h)\,dm(x) =
\iint_{\R^n\times\R^n} \frac{|f(x)-f(y)|^p}{|x-y|^{\alpha p + n}}\,dm(x)\,dm(y).
\end{equation}
This fact motivates the definition of the $\dot{B}_{p,p}^\alpha$-seminorm 
over domains in $\R^n$. Given an open set $\Omega\in\R^n$, one sets
\begin{equation}\label{eqdh4}
\|f\|_{\dot{B}_{p,p}^\alpha(\Omega)}^p=
\iint_{(x,y)\in \Omega^2}\frac{|f(x)-f(y)|^p}{|x-y|^{\alpha p + n}} \,dm(x)\,dm(y),
\end{equation}
and $\|f\|_{{B}_{p,p}^\alpha(\Omega)}= \|f\|_{L^p(\Omega)}+\|f\|_{\dot{B}_{p,p}^\alpha(\Omega)}$.
See \cite{Dispa}.
Analogously, if $\Gamma$ is a chord arc curve or a Lipschitz graph, one defines
\begin{equation}\label{eqdh5}
\|f\|_{\dot{B}_{p,p}^\alpha(\Gamma)}^p=
\iint_{(x,y)\in \Gamma^2}\frac{|f(x)-f(y)|^p}{|x-y|^{\alpha p + 1}} \,d\HH^1(x)\,d\HH^1(y),
\end{equation}
and $\|f\|_{{B}_{p,p}^\alpha(\Gamma)}= \|f\|_{L^p(\HH^1\rest\Gamma)}+\|f\|_{\dot{B}_{p,p}^\alpha(\Gamma)}$.

Concerning the Besov spaces of regularity $1<\alpha<2$, let us remark that, for $f\in L^1_{loc}(\R)$, 
\begin{equation}\label{eqn67}
\|f\|_{\dot{B}_{p,q}^\alpha}^p \approx \|f'\|_{\dot{B}_{p,q}^{\alpha-1}}^p,
\end{equation}
where $f'$ is the distributional derivative of $f$. Further we will use a  
characterization in terms of the coefficients $\beta_1$ due to Dorronsoro.
Recall the definition in \rf{defbeta}.
 In \cite[Theorems 1 and 2]{Dorronsoro2} it is shown
 that, for $1\leq\alpha<2$ and $1\leq p,q<\infty$, one has:
$$\|f\|_{\dot{B}_{p,q}^\alpha} \approx \left(\int_0^\infty \left(h^{-\alpha+1} 
\|\beta_1(f,I(\cdot,h))\|_p\right)^q\,\frac{dh}h\right)^{1/q}.$$
Again, this comparability should be understood modulo polynomials, unless $f$ is compactly supported, say.
In the case $p=q$, an equivalent statement is the following:
$$\|f\|_{\dot{B}_{p,p}^\alpha}^p \approx \biggl(\sum_{I\in\DD(\R)}  
\biggl(\frac{\beta_1(f,I)}{\ell(I)^{\alpha-1}}\biggr)^p\,\ell(I)\biggr)^{1/p}.$$
For other dimensions $n\neq1$ and other indices $\alpha\not\in [1,2)$, there are analogous results which involve approximation by polynomials
of a fixed degree instead of affine functions, which we skip for the sake of 
simplicity. 

Let us remark that the coefficients $\beta_1(f,I)$ are not introduced explicitly in \cite{Dorronsoro2}, and
instead a different notation is used there. 

Finally we recall the relationship between the seminorms $\|\cdot\|_{\dot W^{\alpha,p}(\Omega)}$
and $\|\cdot\|_{\dot B^\alpha_{p,p}(\Omega)}$ in Lipschitz domains. 
We have
$$\|f\|_{\dot W^{\alpha,p}(\Omega)}\lesssim \|f\|_{\dot B^\alpha_{p,p}(\Omega)}
\qquad \mbox{if $1< p\leq2$,}$$
and 
$$\|f\|_{\dot B^\alpha_{p,p}(\Omega)}\lesssim \|f\|_{\dot W^{\alpha,p}(\Omega)}
\qquad \mbox{if $2\leq p<\infty$.}$$


\section{Auxiliary lemmas}\label{sec3}

\begin{lemma}\label{lemanorm}
Let $A:\R\to\R$ be a Lipschitz function with $\|A'\|_\infty\leq c_0$ and $\Gamma\subset\C$ its graph. Denote by $N_0(x)$ the unit
normal of $\Gamma$ at $(x,A(x))$ (whose vertical component is negative, say), which is defined a.e.
Then, 
\begin{equation}\label{eqn62}
|\Delta_h (A')(x)|\approx |\Delta_h N_0(x)|,
\end{equation}
with constants depending on $c_0$.
Thus, for $1\leq p<\infty$ and $0<\alpha<1$,
\begin{equation}\label{eqn63}
\|A\|_{\dot{B}_{p,p}^{\alpha+1}}\approx \|A'\|_{\dot{B}_{p,p}^{\alpha}}\approx\|N_0\|_{\dot{B}_{p,p}^{\alpha}},
\end{equation}
with constants depending on $\alpha$ and $p$, and also on $c_0$ in the second estimate.
\end{lemma}

Above, we set
$$\|N_0\|_{\dot{B}_{p,p}^{\alpha}}:= 
\|N_{0,1}\|_{\dot{B}_{p,p}^{\alpha}} + \|N_{0,2}\|_{\dot{B}_{p,p}^{\alpha}},$$
where $N_{0,i}$, $i=1,2$, are the components of $N_0$.

\begin{proof}
Notice that the first estimate in \rf{eqn63}
is just a restatement of \rf{eqn67}, and the second one follows from \rf{eqn62} and 
the characterization of $\dot{B}_{p,p}^{\alpha}$ in terms of differences in \rf{eqgg0}. So we only have to prove \rf{eqn62}. 

Recall that 
$$N_0(x) = (1+A'(x)^2)^{-1/2}\,\bigl(A'(x),-1\bigr).$$
We will show first the inequality $|\Delta_h N(x)|\lesssim |\Delta_h (A')(x)|$.
Notice that, for arbitrary functions $f,g:\R\to\R$ and $h>0$,
\begin{equation}\label{eqgg0.5}
\Delta_h(f\,g)(x) = f(x)\,\Delta_h g(x) + g(x+h)\,\Delta_h f(x),
\end{equation}
and thus
\begin{equation}\label{eqgg1}
|\Delta_h(f\,g)| \leq \|f\|_\infty\,|\Delta_h g| + \|g\|_\infty\,|\Delta_h f|.
\end{equation}
Also, it is easy to check that 
$$\Delta_h\biggl(\frac1f\biggr)(x) = \frac{-\Delta_h f(x)}{f(x+h)\,f(x)},$$
and so
\begin{equation}\label{eqgg2}
\left|\Delta_h\biggl(\frac1f\biggr)\right| \leq \left\|\,\frac1f\,\right\|_\infty^2 \,|\Delta_h f|.
\end{equation}
On the other hand,
$$\Delta_h\left(\sqrt{1+f^2}\right)(x) = \frac{\bigl(f(x+h)+f(x)\bigr)\,\,\Delta_h f(x)}{
\sqrt{1+f(x+h)^2} +\sqrt{1+f(x)^2}},$$
and thus it follows that
\begin{equation}\label{eqgg3}
\left|\Delta_h\left(\sqrt{1+f^2}\right)(x)\right| \leq |\Delta_h f(x)|.
\end{equation}

From \rf{eqgg2} and \rf{eqgg3} we infer that 
\begin{equation}\label{eqgg4}
|\Delta_h N_{0,2}(x)|=
\left|\Delta_h\left((1+A'(x)^2)^{-1/2}\right)\right| \leq |\Delta_h (A')(x)|.
\end{equation}
 Also, from \rf{eqgg4} and \rf{eqgg1}, taking into account that 
$\|A'\|_\infty\leq c_0$, we deduce that 
$$|\Delta_h N_{0,1}(x)|= \left|\Delta_h\left(A'(x)\,(1+A'(x)^2)^{-1/2}\right)\right|\leq (c_0+1)\,|\Delta_h (A')(x)|.$$

Let us see now that $|\Delta_h (A')(x)|\lesssim |\Delta_h N_0(x)|$.
 From \rf{eqgg2}, we infer that
$$\left|\Delta_h\left(\sqrt{1+A'(x)^2}\right)\right| \leq (1+c_0^2)\,|\Delta_h (N_{0,2})(x)|.$$
Finally, since $A'=N_{0,1}\sqrt{1+(A')^2}$, using \rf{eqgg1} we get
\begin{align*}
|\Delta_h(A')(x)|&\leq \textstyle\sqrt{1+c_0^2}\,|\Delta_h(N_{0,1})(x)| + \left|\Delta_h\left(\sqrt{1+A'(x)^2}\right)\right|\\
&\leq \textstyle\sqrt{1+c_0^2}\,|\Delta_h(N_{0,1})(x)|+ (1+c_0^2)\,|\Delta_h (N_{0,2})(x)|,
\end{align*}
as wished.
\end{proof}

\begin{rem}
From the characterization of Besov spaces in terms of differences, it turns
out that 
if $N(z)$ stands for the unit normal at $z\in\Gamma$ (with a suitable orientation), 
then
$$\|N_0\|_{\dot{B}_{p,p}^\alpha}\approx\|N\|_{\dot{B}_{p,p}^\alpha(\Gamma)}$$
for $1\leq p<\infty$ and $0<\alpha<1$.
\end{rem}

Recall that in \rf{defbeta} we defined the coefficients $\beta_1$ associated to
a function $f$. Now we introduce an analogous notion replacing 
$f$ by a chord arc curve $\Gamma$ (which may be the boundary
of a Lipschitz domain). Given $P\in\DD(\Gamma)$, we set
\begin{equation}\label{defbetag} 
\beta_1(\Gamma,P) = \inf_L \frac1{\ell(P)}\int_{3P} \frac{\dist(x,L)}{\ell(P)}\,
d\HH^1(x),
\end{equation}
where the infimum is taken over all the lines $L\subset\C$.

Next lemma is a direct consequence of our previous results and the characterization of
homogeneous Besov spaces in terms of the $\beta_1$'s from Dorronsoro, described
in the preceding section.

\begin{lemma}\label{lemdorron}
Let $\Omega$ be a Lipschitz domain. Suppose that the outward unit normal satisfies
$N\in \dot{B}_{p,p}^\alpha(\partial\Omega)$, for some $1\leq p<\infty$, $0<\alpha<1$.
Then,
$$\sum_{P\in\DD(\partial\Omega)}  
\biggl(\frac{\beta_1(\partial\Omega,P)}{\ell(P)^{\alpha}}\biggr)^p\,\ell(P)
\lesssim \|N\|_{\dot{B}_{p,p}^\alpha(\partial\Omega)}^p + c\,\HH^1(\partial\Omega)^{1-\alpha\,p}.
$$
with $c$ depending on $\HH^1(\partial\Omega)/R$.
\end{lemma}

\begin{proof}
Let $\delta,R>0$ be such that $\Omega$ is a $(\delta,R)$-Lipschitz domain.
Consider  a finite covering of $\partial\Omega$ by a family of balls $\{B(x_i,R/4)\}_{1\leq i \leq m}$, with $x_i\in\partial \Omega$. Notice that for any cube $P\in\DD(\partial \Omega)$ with
$\ell(3P)< R/4$ there exists some ball $B(x_i,R/2)$ containing $P$.
Thus, to prove the lemma it is enough to see
that, for each $i$,
\begin{equation}\label{eqdh566}
\sum_{P\in\DD(\partial\Omega):P\subset B(x_i,R/2)} 
\biggl(\frac{\beta_1(\partial\Omega,P)}{\ell(P)^{\alpha}}\biggr)^p\,\ell(P)
\lesssim \|N\|_{\dot{B}_{p,p}^\alpha(\partial\Omega)}^p+\HH^1(\partial\Omega)^{1-\alpha\,p}.
\end{equation}

So fix $i$ with $1\leq i \leq m$ and let $A:\R\to\R$ be a Lipschitz functions such that, after
a suitable rotation, 
$$\Omega\cap B(x_i,R)=\{(x,y)\in B(x_i,R):\,y>A(x)\}.$$
Moreover we may assume that $A(x_i)=0$ and that $\supp A\subset [-2R,2R]$.
Let $\vphi:\R\to\R$ be a $\CC^\infty$ function which equals $1$ on $[-R/2,R/2]$ and
vanishes on $\C\setminus [-3R/4,3R/4]$. Consider the function $\wt A=\vphi \,A$. 
From \rf{eqgg0.5} and \rf{eqn62} we deduce that
$$
|\Delta_h(\wt A')| \leq \vphi |\Delta_h A'| + \|A'\|_\infty\,|\Delta_h \vphi|
\leq \chi_{[-3R/4,3R/4]} \,\bigl|\Delta_h N(x,A(x))\bigr| + c\,|\Delta_h \vphi|
.$$
Notice also that, for $|h|\leq R/4$,
$$\supp\Bigl(\chi_{[-3R/4,3R/4]} \,\Delta_h N(\cdot,A(\cdot))\Bigr) \subset [-R,R].$$
As a consequence, $(x,A(x))\in\partial\Omega$ for $x$ belonging to the support above, and so 
we get 
\begin{align*}
\iint_{|h|\leq  R/4}
\frac{|\Delta_h(\wt A')|^p}{|h|^{\alpha p + 1}} \,dx\,dh & \lesssim 
\iint_{(\partial\Omega)^2}
\frac{|N(x)-N(y)|^p}{|x-y|^{\alpha p + 1}} \,d\HH^1(x)\,d\HH^1(y)\\
& \quad +
\iint
\frac{|\Delta_h\vphi|^p}{|h|^{\alpha p + 1}} \,dx\,dh \\
&\approx
\|N\|_{\dot{B}_{p,p}^\alpha(\partial\Omega)}^p + \|\vphi\|_{\dot{B}_{p,p}^\alpha}^p
.
\end{align*}
It is easy to check that 
$$\|\vphi\|_{\dot{B}_{p,p}^\alpha}^p\lesssim R^{1-\alpha\,p}.$$
Taking into account that $\|(\wt A)'\|_\infty\leq c$ and that $\wt A'$ vanishes out of $[-R,R]$, we deduce
\begin{align*}
\iint_{|h|>  R/4}
\frac{|\Delta_h(\wt A')|^p}{|h|^{\alpha p + 1}} \,dx\,dh& \lesssim
\int_{|x|\leq R} \int_{|h|>  R/4}
\frac{1}{|h|^{\alpha p + 1}} \,dx\,dh \\
&\quad + 
\int_{|x+h|\leq R} \int_{|h|>  R/4}
\frac{1}{|h|^{\alpha p + 1}} \,dx\,dh \\
&\lesssim R^{1-\alpha\,p}.
\end{align*}
Therefore, $ \|\wt A\|_{\dot{B}_{p,p}^{\alpha+1}}^p\approx\|(\wt A)'\|_{\dot{B}_{p,p}^\alpha}^p\lesssim \|N\|_{\dot{B}_{p,p}^\alpha(\partial\Omega)}^p+ R^{1-\alpha\,p}$.
Thus, from Dorronsoro's theorem, we  get
$$\sum_{Q\in\DD(\R)} 
\biggl(\frac{\beta_1(\wt A,Q)}{\ell(Q)^{\alpha}}\biggr)^p\,\ell(Q)
\lesssim \|N\|_{\dot{B}_{p,p}^\alpha(\partial\Omega)}^p+ R^{1-\alpha\,p}.$$
Since the graph of $\wt A$ coincides with $\partial \Omega$ on $B(x_i,R) \cap 
\bigl([-R/2,R/2]\times\R\bigr)$, we get
\begin{align*}
\sum_{P\in\DD(\partial\Omega):P\subset B(x_i,R/2)} 
\biggl(\frac{\beta_1(\partial\Omega,P)}{\ell(P)^{\alpha}}\biggr)^p\,\ell(P)
&\lesssim \|N\|_{\dot{B}_{p,p}^\alpha(\partial\Omega)}^p+R^{1-\alpha\,p}\\& =
\|N\|_{\dot{B}_{p,p}^\alpha(\partial\Omega)}^p+c_1\,\HH^1(\partial \Omega)^{1-\alpha\,p},
\end{align*}
with $c_1= R^{1-\alpha\,p}/\HH^1(\partial\Omega)^{1-\alpha\,p}$. So \rf{eqdh566} holds and we
are done.
\end{proof}

\vvv
\begin{rem}\label{rem11}
Given $0<\alpha<1$, for a Lipschitz domain, from the definition \rf{eqdh5}, it is easy to deduce that
$$
\|N\|_{\dot{B}_{p,p}^\alpha(\partial\Omega)}^p \gtrsim\HH^1(\partial\Omega)^{1-\alpha\,p}.$$
So, in fact we have
$$\sum_{P\in\DD(\partial\Omega)}  
\biggl(\frac{\beta_1(\partial\Omega,P)}{\ell(P)^{\alpha}}\biggr)^p\,\ell(P)
\lesssim \|N\|_{\dot{B}_{p,p}^\alpha(\partial\Omega)}^p.
$$
\end{rem}



\vvv
\section{Preliminary lemmas for the Theorem \ref{teo}}\label{sec4}

Let $\Omega\subset\C$ be an open set. If $\Omega$ has finite Lebesgue measure, then
\begin{equation}\label{eqbeurling21}
B\chi_\Omega(z)=\lim_{\ve\to0}\frac{-1}\pi\int_{|z-w|>\ve}\frac{1}{(z-w)^2}
\chi_\Omega(w)\,dm(w).
\end{equation}
Otherwise, $B(\chi_\Omega)$ is a BMO function and,
thus, it is defined modulo constants. Actually, a possible way to assign a precise value to $B(\chi_
\Omega)(z)$
is the following:
\begin{equation}\label{eqbeurling22}
B\chi_\Omega(z)=\lim_{\ve\to0}\frac{-1}\pi\int_{|z-w|>\ve}\left(\frac{1}{(z-w)^2} - 
\frac1{(z_0-w)^2}\right)
\chi_\Omega(w)\,dm(w),
\end{equation}
where $z_0$ is some fixed point, with $z_0\not\in\overline\Omega$, for example. It is easy to check
that the preceding principal value integral exists for all $z\in\C$ and that, moreover, it
is analytic in $\C\setminus \partial\Omega$.

\begin{lemma}
Let $\Omega\subset\C$ be an open set.
The function $B(\chi_\Omega)$ is analytic in $\C\setminus\partial\Omega$ and moreover, for any $z\in
\C\setminus\partial\Omega$, we have
\begin{equation}\label{eq**}
\partial B(\chi_\Omega)(z) = \frac2\pi \int_{|z-w|>\ve}\frac1{(z-w)^3}\,\chi_{\Omega}(w)\,dm(w),
\end{equation}
for $0<\ve<\dist(z,\partial \Omega)$. 
\end{lemma}

When $\Omega$ has infinite measure, saying that $B(\chi_\Omega)$ is analytic in $\C\setminus\partial\Omega$ means that the function defined in \rf{eqbeurling22} is analytic for each choice of $z_0$.
Notice that, in any case, the derivative $\partial B(\chi_\Omega)$ is independent of $z_0$.

\begin{proof}
It is easy to check that $B(\chi_\Omega)$ is analytic in $\C\setminus\overline\Omega$ and that its $\partial$ derivative equals \rf{eq**}. This follows by differentiating
under the integral on the right side of \rf{eqbeurling21} or \rf{eqbeurling22}, for
$0<\ve<\dist(z,\partial\Omega)$.

It remains to show that, in $\Omega$, $\bar\partial B\chi_\Omega=0$ and that 
\rf{eq**} also holds. For a fixed $z\in\Omega$  and for $0<\delta\leq \ve<\dist(z,\partial \Omega)$
notice that
$$\int_{\delta\leq|z-w|\leq\ve}\frac1{(z-w)^2}\,dm(w)=0.$$
 As a consequence,
$$B\chi_\Omega(z)=\frac{-1}\pi\int_{|z-w|>\delta}\frac{1}{(z-w)^2}\,
\chi_\Omega(w)\,dm(w)$$
or, in the case where $m(\Omega)=\infty$,
$$
B\chi_\Omega(z)=\frac{-1}\pi\int_{|z-w|>\delta}\left(\frac{1}{(z-w)^2} - 
\frac1{(z_0-w)^2}\right)
\chi_\Omega(w)\,dm(w).$$
Let $\vphi$ be a $\CC^\infty$ radial
function  which vanishes on $B(0,\ve/2)$ and equals $1$ on $\C\setminus B(0,\ve)$.
From the preceding identities, writing $\vphi$ as a convex combination of  functions
of the form $\chi_{\C\setminus B(0,\delta)}$, one deduces that
$$B\chi_\Omega(z)=\frac{-1}\pi\int\frac{\vphi(z-w)}{(z-w)^2}\,
\chi_\Omega(w)\,dm(w)$$
or, analogously,
$$
B\chi_\Omega(z)=\frac{-1}\pi\int\left(\frac{\vphi(z-w)}{(z-w)^2} - 
\frac1{(z_0-w)^2}\right)
\chi_\Omega(w)\,dm(w).$$
In any case, it is straightforward to check that one can differentiate under the integral
sign and thus
$$\partial B\chi_\Omega(z)=\frac{-1}\pi\int\partial \left(\frac{\vphi(z-w)}{(z-w)^2}\right)\,
\chi_\Omega(w)\,dm(w).$$ 
The same identity holds replacing the $\partial$ derivative by the $\bar\partial$ one. So,
$$\partial B\chi_\Omega(z)=\frac{-1}\pi\,\partial\Bigl(\frac{1}{w^2}\,\vphi\Bigr)*\chi_\Omega(z)
\quad \mbox{ and }\quad
\bar\partial B\chi_\Omega(z)=\frac{-1}\pi\,\bar\partial\Bigl(\frac{1}{w^2}\,\vphi\Bigr)*\chi_\Omega
(z).$$

We write $\vphi(w)=\psi(|w|^2)$, and then we get
$$\partial\Bigl(\frac{1}{w^2}\,\vphi\Bigr) = \frac{-2}{w^3}\,\psi(|w|^2) + \frac1{w^2}\,
\psi'(|w|^2)\overline w = \frac{-2}{w^3}\,\psi(|w|^2) + \frac{|w|^2} {w^3}\,
\psi'(|w|^2)=: \frac{-2}{w^3}\,\wt\vphi(z),$$ 
where $\wt \vphi$ is 
$\CC^\infty$, radial, vanishes on $B(0,\ve/2)$ and equals $1$ on $\C\setminus B(0,\ve)$.
Arguing as above, it turns out that
$$\Bigl(\frac{-2}{w^3}\,\wt\vphi\Bigr)*\chi_\Omega (z) = 
\Bigl(\frac{-2}{w^3}\,\chi_{\C\setminus B(0,\ve)}\Bigr)*\chi_\Omega (z),$$
and then \rf{eq**} follows.

On the other hand, we have
$$\bar \partial\Bigl(\frac{1}{w^2}\,\vphi\Bigr) =  \frac1{w^2}\,
\psi'(|w|^2) \,w = \frac{\psi'(|w|^2)}w.$$ 
Since $\supp(\psi')\subset B(0,\ve)$, we derive
$$\bar\partial B\chi_\Omega(z)=
\frac{-1}\pi\,\frac{\psi'(|w|^2)}w*\chi_\Omega(z)=
\frac{-1}\pi \int \frac{\psi'(|w|^2)}w\,dm(w).$$
Using polar coordinates, say, it is easy to check that the last integral vanishes. So
$\bar\partial B\chi_\Omega(z)=0$. This means that $ B\chi_\Omega$ is analytic in $\Omega$.
\end{proof}

\vvv

\begin{lemma}\label{lemmaPi}
Let $\Pi\subset\C$ be a half plane. Then 
$\partial B(\chi_\Pi)=0$ in $\C\setminus\partial\Pi$. Equivalently,
for all $z\not\in\partial\Pi$ and $0<\ve<\dist(z,\partial\Pi)$, we have
$$\int_{|z-w|>\ve}\frac1{(z-w)^3}\,\chi_{\Pi}(w)\,dm(w) = 0.$$
\end{lemma}

\begin{proof}
That the two statements above are equivalent is a direct consequence of the preceding lemma.
Let us prove the second one. To this end, assume for simplicity that $\Pi=\{w\in\C:\imag(w)>0\}$.

Let $B_r$ be a ball with radius $r$ centered at $r\,i$. 
It is known that $B(\chi_{B_r})$ vanishes identically 
on $B_r$ and equals $-1/(\pi\,(z-r\,i)^2)$ out of $B$ (this can be deduced by computing the Cauchy
transform of $\chi_B$ and then applying the $\partial$ derivative).
Therefore,  for $z\not\in\R$ with $|\imag z|>\ve$, from the preceding lemma we infer that if
$r$ is big enough
$$\int_{|z-w|>\ve}\frac1{(z-w)^3}\,\chi_{B_r}(w)\,dm(w)=\left\{
\begin{array}{ll}
0& \mbox{if $z\in B_r$},\\ \\
\dfrac2{\pi\,(z-r\,i)^3}& \mbox{if $z\not\in B_r$.}
\end{array}
\right.
$$
Letting $r\to\infty$, since $\chi_{B_r}(w)\to\chi_\Pi(w)$ a.e.\ $w\in\C$,
by the dominated convergence theorem, we are done.
\end{proof}

\vvv


In the remaining of the paper, to simplify notation, for $Q\in\DD(\partial\Omega)$ we will denote
$$\beta_1(Q)\equiv\beta_1(\partial \Omega,Q).$$

\vvv

\begin{lemma}\label{lemdifsim}
Let $\Omega$ be either a Lipschitz  domain or a special Lipschitz domain,  and consider $Q\in\DD(\partial \Omega)$ and a ball $B_r$ centered at some point
from $Q$, with radius $\ell(Q)\leq r\leq \theta\,\diam(\Omega)$, 
with $\theta=\theta(\Omega)>0$ small enough. 
Let $L_Q$ be a line  that minimizes $\beta_1(Q)$.
Let $\Pi_Q$ be a half plane such that $\partial\Pi_Q=L_Q$ and suppose that there exists some point
$z_Q\in\Pi_Q\cap\Omega\cap B_r$ such that $\dist(z_Q,L_Q)=\frac12\,r$. 
Then
\begin{equation}\label{eqbeta4}
m(B_r\cap(\Omega\Delta\Pi_Q))\leq c\sum_{P\in\DD(\partial \Omega):Q\subset P,\ell(P)\leq Mr}
\beta_1(P)\,r^2,
\end{equation}
assuming that $M$ has been chosen big enough (depending on the Lipschitz character of $\Omega$).
\end{lemma}

\vvv
The condition on the existence of the point $z_Q$ tells which of the half planes whose boundary is
$L_Q$ is the selected one for \rf{eqbeta4}. The constant $\theta$ is superfluous when $\Omega$ is a 
special Lipschitz domain, since $\diam(\Omega)=\infty$ in this case.

\vvv

\begin{proof} Suppose that $\Omega$ is a Lipschitz domain.
Let $R\in\DD(\partial\Omega)$ be such that $B_{2r}\cap\Omega\subset 2R$ and $\ell(R)\approx r$.
It is enough to show that 
$$m(B_r\cap(\Omega\Delta\Pi_Q))\leq c\sum_{P\in\DD(\partial \Omega):Q\subset P\subset R}
\beta_1(P)\,\ell(R)^2.$$
Moreover, we may assume that 
\begin{equation}\label{eqassum}
\sum_{P\in\DD(\partial \Omega):Q\subset P\subset R}
\beta_1(P)\leq\ve_0,
\end{equation}
with $\ve_0$ small enough. Otherwise, the estimate is trivial.

Taking $\delta$ small enough, after a rotation if necessary, 
we may also assume that $\partial \Omega\cap B_{10r}$ is given by the graph of a 
Lipschitz function $y=A(x)$ intersected with $B_{10r}$, and that 
$$\Omega\cap B_{10r}=\{(x,y)\in B_{10r}:y>A(x)\}.$$

Let $L_R$ be a line that minimizes $\beta_1(R)$. By the assumption \rf{eqassum}, we know that
$\partial\Omega\cap B_{2r}$ is very close to $L_R$.
Further, it is easy to check that
$$\dist_H(L_Q\cap B_{3r},L_R\cap B_{3r})\leq c\sum_{P\in\DD(\partial \Omega):Q\subset P\subset R}
\beta_1(P)\,\ell(R),$$
where $\dist_H$ stands for the Hausdorff distance.
Thus, if $\ve_0$ is taken small enough, then $\partial\Omega\cap B_{3r}\subset 
U_{r/100}(L_Q)$, where $U_d(A)$ stands for the $d$-neighborhood of $A$.

It easily follows that for $z\in \partial\Omega\cap B_{2r}$,
$$\dist(z,L_Q)=\dist(z,L_Q\cap B_{3r}),\qquad \dist(z,L_R)=\dist(z,L_R\cap B_{3r}).$$
We deduce
\begin{align*}
m(B_r\cap(\Omega\Delta\Pi_Q)) & \leq \int_{\partial\Omega\cap B_{2r}}\dist(z,L_Q)\,d\HH^1(z)\\
& \leq \int_{\partial\Omega\cap B_{2r}}\bigl(\dist(z,L_R)+ \dist_H(L_R\cap B_{3r},L_Q\cap B_{3r})\bigr)\,d\HH^1(z)\\
& \leq \beta_1(R)\ell(R)^2 + c\sum_{P\in\DD(\partial \Omega):Q\subset P\subset R}
\beta_1(P)\,\ell(R)^2,
\end{align*}
which proves the lemma, for $\Omega$ being a Lipschitz domain.

If $\Omega$ is a special Lipschitz domain, the proof is analogous. The details are left for the reader.
\end{proof}

\vvv


\section{Proof of Theorem \ref{teo}}\label{sec5}

First we suppose that $\Omega$ is a (bounded) Lipschitz domain. 
Consider a decomposition of $\Omega$ into a family $\WW(\Omega)$ of Whitney squares 
as explained in Subsection \ref{sec2.1},
so that they have disjoint interiors,  $\Omega = \bigcup_{Q\in \WW(\Omega)}Q$,  \,
$\sum_{Q\in \WW(\Omega)}\chi_{5Q}
\leq c_1$, and, moreover, $\rho\,Q\cap\partial \Omega\neq\varnothing$.
In fact, we have
$\dist(Q,\partial\Omega)\approx \ell(Q)$ for $Q\in \WW(\Omega)$.
Recall also that to each
$Q\in\WW(\Omega)$ we assign a cube $\phi(Q)\in\DD(\partial \Omega)$ such that $\phi(Q)\cap\rho Q\neq
\varnothing$ and $\diam(\phi(Q))\approx\ell(Q)$.

We write
\begin{equation}\label{eqdbeur}
\|\partial B\chi_\Omega\|_{L^p(\Omega)}^p = \sum_{Q\in \WW(\Omega)}\int_Q|\partial B\chi_\Omega|^p\,dm.
\end{equation}
Our first task consists in estimating $\partial B\chi_\Omega(z)$ for $z$ belonging 
to $Q\in \WW(\Omega)$. To this end, consider a line $L_Q$ that minimizes
$\beta_1(\phi(Q))$.
We claim that
\begin{equation}\label{eqclaim1}
\bigl|\partial B\chi_\Omega(z)\bigr|\leq c_3\sum_{R\in\DD(\partial\Omega):R\supset\phi(Q)}\frac{\beta_1(R)}{\ell(R)} + c_3\,\frac1{\diam(\Omega)}.
\end{equation}
To prove this estimate we may assume that $\beta_1(\phi(Q))\leq\ve_0$, with $\ve_0>0$ small enough.
Indeed, from \rf{eq**} it turns out that $\bigl|\partial B\chi_\Omega(z)\bigr|\leq c/\ell(Q)$, by choosing $\ve
=\ell(Q)$ there, and so \rf{eqclaim1} holds if $\beta_1(\phi(Q))>\ve_0$, with $c_3=c\ve_0^{-1}$.

So suppose that $\beta_1(\phi(Q))\leq\ve_0$, with $\ve_0$ very small. In this case, $L_Q$ 
is very close to 
$\partial\Omega$ near $\phi(Q)$, and then one infers that 
$$\dist(z,L_Q)\approx\ell(Q).$$
Denote by $\Pi_Q$ the half plane whose boundary is $L_Q$ and contains $z$. Take $0<\ve<\dist(z,\partial\Omega)$.
Since $(\frac1{z^3}\chi_{B(0,\ve)^c})*\chi_{\Pi_Q}$ vanishes on $\Pi_Q\ni z$, we have
$$\bigl|\partial B\chi_\Omega(z)\bigr|= \Bigl|\Bigl(\frac{2\pi}{z^3}\chi_{B(0,\ve)^c}\Bigr)*(\chi_\Omega- \chi_{\Pi_Q})(z)\Bigr|\leq 
\frac{2\pi}{|z|^3}*\chi_{\Omega\Delta\Pi_Q}(z).$$
For each $n\geq0$, let $B_n$ be a ball centered at $z'\in\phi(Q)$ with 
$$\diam(B_n)= 2^n\diam(\phi(Q))\approx 2^n\ell(Q).$$
Set also $B_{-1}=\varnothing$ and take $N$ such that $\frac12\theta\,\diam(\Omega)<\diam(B_N)\leq
\theta\,\diam(\Omega)$, with $\theta$ from Lemma \ref{lemdifsim}. Then we write
\begin{align*}
\frac{2\pi}{|z|^3}*\chi_{\Omega\Delta\Pi_Q}(z) & = 
\sum_{n=0}^N\frac{2\pi}{|z|^3} * \chi_{B_n\cap(\Omega\Delta\Pi_Q)}(z) + \frac{2\pi}{|z|^3} * 
\chi_{B_N^c\cap(\Omega\Delta\Pi_Q)}(z)\\
&\leq c
\sum_{n=0}^N\frac{1}{\ell(2^nQ)^3} \,m(B_n\cap(\Omega\Delta\Pi_Q)) + \frac{c}{\diam(\Omega)}.
\end{align*}
 By Lemma \ref{lemdifsim}, for $0\leq n\leq N$, we have
$$m(B_n\cap(\Omega\Delta\Pi_Q))\leq c\sum_{P\in\DD(\partial \Omega):\phi(Q)\subset P\subset R}
\beta_1(P)\diam(R)^2,$$
where $R\in\DD(\partial\Omega)$ is the biggest cube containing $\phi(Q)$ such that $\ell(R)
\leq \frac12\,M\,\diam(B_n)$, with $M$ from Lemma \ref{lemdifsim}. In particular, it turns out that $\ell(R)\approx\diam(B_n)$.
Then we obtain
\begin{align}\label{nucl}
\frac{2\pi}{|z|^3}*\chi_{\Omega\Delta\Pi_Q}(z) & \leq c
\sum_{R\in\DD(\partial\Omega):R\supset \phi(Q)}\frac{1}{\ell(R)^3} \,
\sum_{P\in\DD(\partial \Omega):\phi(Q)\subset P\subset R}
\beta_1(P)\ell(R)^2 + \frac{c}{\diam(\Omega)}\\
& = c
\sum_{P\in\DD(\partial \Omega):P\supset \phi(Q)}
\beta_1(P) 
\sum_{R\in\DD(\partial\Omega):R\supset P}\frac1{\ell(R)} + \frac{c}{\diam(\Omega)}\nonumber\\
&\leq c\sum_{P\in\DD(\partial\Omega):P\supset\phi(Q)}\frac{\beta_1(P)}{\ell(P)} + c\,\frac1{\diam(\Omega)},\nonumber
\end{align}
which proves our claim \rf{eqclaim1}.

Plugging \rf{eqclaim1} into \rf{eqdbeur}, we get
$$\|\partial B\chi_\Omega\|_{L^p(\Omega)}^p \lesssim \sum_{Q\in \WW(\Omega)}\biggl(
\sum_{P\in\DD(\partial\Omega):P\supset\phi(Q)}\frac{\beta_1(P)}{\ell(P)} \biggr)^pm(Q)+ \frac{m(\Omega)}
{\diam(\Omega)^p}.$$
The last term on the right side is bounded by $c/\diam(\Omega)^{p-2}$. For the first one we use Cauchy-Schwarz, and then we 
get
\begin{align*}
\biggl(
\sum_{P\in\DD(\partial\Omega):P\supset\phi(Q)}\!\!\frac{\beta_1(P)}{\ell(P)} \biggr)^p\! &
\leq \biggl(
\sum_{P\in\DD(\partial\Omega):P\supset\phi(Q)}\!\frac{\beta_1(P)^p}{\ell(P)^{p-1/2}} \!\biggr)
\biggl(
\sum_{P\in\DD(\partial\Omega):P\supset\phi(Q)}\frac{1}{\ell(P)^{p'/(2p)}} \biggr)^{p/p'}
\\
&\lesssim
\sum_{P\in\DD(\partial\Omega):P\supset\phi(Q)}\frac{\beta_1(P)^p}{\ell(P)^{p-1/2}} \frac1{\ell(\phi(Q))^{1/2}}.
\end{align*}
Thus,
\begin{align*}
\sum_{Q\in \WW(\Omega)}\biggl(
\sum_{P\in\DD(\partial\Omega):P\supset\phi(Q)}\frac{\beta_1(P)}{\ell(P)} \biggr)^pm(Q) &\lesssim
\sum_{Q\in \WW(\Omega)}
\sum_{P\in\DD(\partial\Omega):P\supset\phi(Q)}\frac{\beta_1(P)^p}{\ell(P)^{p-1/2}}\, \ell(\phi(Q))^{3/2}\\
& = 
\sum_{P\in\DD(\partial\Omega)}\frac{\beta_1(P)^p}{\ell(P)^{p-1/2}}\, \sum_{Q\in \WW(\Omega):\phi(Q)\subset P}
\ell(\phi(Q))^{3/2}.
\end{align*}
Notice that
$$\sum_{Q\in \WW(\Omega):\phi(Q)\subset P}
\ell(\phi(Q))^{3/2}\lesssim\sum_{\wt Q\in \DD(\partial\Omega):\wt Q\subset P}
\ell\bigl(\wt Q\bigr)^{3/2}\lesssim\ell(P)^{3/2},$$
and so
$$\|\partial B\chi_\Omega\|_{L^p(\Omega)}^p \lesssim
\sum_{P\in\DD(\partial\Omega)}\frac{\beta_1(P)^p}{\ell(P)^{p-2}}
+ \frac 1{\diam(\Omega)^{p-2}}.$$

Observe now that the sum on the right side can be written as
\begin{equation}
\sum_{P\in\DD(\partial\Omega)}\frac{\beta_1(P)^p}{\ell(P)^{p-2}} = \sum_{P\in\DD(\partial\Omega)}\biggl(\frac{\beta_1(P)}{\ell(P)^{1-1/p}}\biggr)^p\,\ell(P).
\end{equation}
By Lemma \ref{lemdorron}, we know that the right side above is bounded by
$c\,\|N\|_{\dot{B}_{p,p}^{1-1/p}(\partial\Omega)}^p + c\,\HH^1(\partial\Omega)^{2-p}$. 
Moreover, as in Remark \ref{rem11}, we have
$$
\|N\|_{\dot{B}_{p,p}^{1-1/p}(\partial\Omega)}^p \gtrsim\HH^1(\partial\Omega)^{2-p}\approx
\diam(\partial\Omega)^{2-p}.$$
So we get
$$\|\partial B\chi_\Omega\|_{L^p(\Omega)}^p \lesssim  
\|N\|_{\dot{B}_{p,p}^{1-1/p}(\partial\Omega)}^p + \HH^1(\partial\Omega)^{2-p}
\approx \|N\|_{\dot{B}_{p,p}^{1-1/p}(\partial\Omega)}^p,$$
as wished.

The arguments for special Lipschitz domains
are analogous, and even easier. Roughly speaking, the only difference is that the terms above 
involving $\diam(\Omega)$,
such as the last term in \rf{eqclaim1}, do not appear.
\fiproof


\section{Preliminary lemmas for the proof of Theorem \ref{teo2}}\label{sec6}

In Section \ref{sec4} we showed that, for any given half plane $\Pi$, $B(\chi_\Pi)$ is analytic in $\C\setminus\partial\Pi$ and
that $\partial B(\chi_\Pi)=0$ in $\C\setminus\partial\Pi$. As a direct consequence, we have:

\begin{lemma}\label{lemmaPi2}
Let $\Pi\subset\C$ be a half plane and let $x,y\in\C$ be in the same component of $\C\setminus
\partial\Pi$. Then, for all $0<\ve<\min\bigl(\dist(x,\partial\Pi),\dist(y,\partial\Pi)\bigr)$,
\begin{equation}\label{dif}
B\chi_\Pi(x) - B\chi_\Pi(y) =
\int \left[\frac{1}{(x-z)^2}\,\chi_{\Pi\setminus B(x,\ve)}(z)-\frac{1}{(y-z)^2}\,\chi_{\Pi\setminus B(y,\ve)}(z) \right] \,dm(z) =0.
\end{equation}
\end{lemma}

\begin{proof}
The first identity in \rf{dif} follows from the definition of $B\chi_\Pi(x)$ and $B\chi_\Pi(y)$,
in the sense of \rf{eqbeurling22}, using also that
$$
\int_{\delta<|x-z|\leq\ve} \frac{1}{(x-z)^2}\,dm(z)=\int_{\delta<|y-z|\leq\ve} \frac{1}{(y-z)^2}\,dm(z)=0$$
for $0<\delta<\ve$. The second identity in \rf{dif} is due to the fact that $B\chi_\Pi$ is constant in 
each component of $\C\setminus\partial\Pi$.
\end{proof}

For two cubes $Q,R$, either from $\DD(\partial\Omega)$ or from $\WW(\Omega)$, we denote
$$D(Q,R)=\ell(Q)+\ell(R)+\dist(Q,R).$$
This is the ``big distance'' between $Q$ and $R$, which will be used below.

\begin{lemma}\label{lemmaCor}
Let $0<\eta<\tau$ and let $\Omega$ be either a Lipschitz or a special Lipschitz domain. Then,
for all $Q\in\DD(\partial\Omega)$ we have
$$\sum_{R\in\DD(\partial\Omega)}\frac{\ell(R)^{1+\eta}}{D(Q,R)^{1+\tau}}\leq\frac{c}{\ell(Q)^{\tau-\eta}}, $$
with $c$ depending on $\eta$ and $\tau$.
\end{lemma}

\begin{proof}
This follows easily from the fact that $\partial\Omega$ has linear growth. Indeed, first notice
that for each $\ell_0>0$,
$$\sum_{R\in\DD(\partial\Omega):\,\ell(R)=\ell_0}\frac{\ell(R)^{1+\eta}}{D(Q,R)^{1+\tau}}
=
\ell_0^\eta\,\sum_{R\in\DD(\partial\Omega):\,\ell(R)=\ell_0}\frac{\ell(R)}{D(Q,R)^{1+\tau}}
\lesssim \frac{\ell_0^\eta}{\max(\ell(Q),\ell_0)^\tau}.$$
Therefore,
\begin{align*}
\sum_{R\in\DD(\partial\Omega)}\frac{\ell(R)^{1+\eta}}{D(Q,R)^{1+\tau}} &
= \sum_{k\in\Z}\sum_{\substack{R\in\DD(\partial\Omega):\\ \ell(R)=2^k\ell(Q)}}\frac{\ell(R)^{1+\eta}}{D(Q,R)^{1+\tau}} 
\lesssim
\sum_{k\in\Z} \frac{2^{\eta k}\,\ell(Q)^\eta}{ \max(1,2^k)^\tau \,\ell(Q)^\tau}\\
&\approx \sum_{k\geq 0} 2^{(\eta-\tau) k}\,\ell(Q)^{\eta-\tau}
+ \sum_{k< 0} 2^{\eta k}\,\ell(Q)^{\eta-\tau} \lesssim \ell(Q)^{\eta-\tau}.
\end{align*}
\end{proof}

We will split the proof of Theorem \ref{teo2} into two parts. The first one deals with the fact that
$B(\chi_\Omega)\in\dot W^{\alpha,p}(\Omega)$:

\begin{lemma}\label{leminc1}
Under the assumptions of Theorem \ref{teo2}, 
for $0<\alpha<1$ and $1<p<\infty$ with $\alpha p>1$,
 we have 
$B(\chi_\Omega)\in\dot W^{\alpha,p}(\Omega)$ and moreover,
$$\|B(\chi_\Omega)\|_{\dot W^{\alpha,p}(\Omega)}\lesssim
\|N\|_{\dot {B}_{p,p}^{\alpha-1/p}(\partial\Omega)}.$$
\end{lemma}

Afterwards,  we will show that
$B(\chi_\Omega)\in\dot B^\alpha_{p,p}(\Omega)$:

\begin{lemma}\label{leminc2}
Under the assumptions of Theorem \ref{teo2}, for $0<\alpha<1$ and $1<p<\infty$ with $\alpha p>1$,
we have
$B(\chi_\Omega)\in\dot B^\alpha_{p,p}(\Omega)$ and
moreover,
$$\|B(\chi_\Omega)\|_{\dot B^\alpha_{p,p}(\Omega)}\lesssim
\|N\|_{\dot {B}_{p,p}^{\alpha-1/p}(\partial\Omega)}.$$
\end{lemma}


We will prove both lemmas in the following sections.


\section{Proof of Lemma \ref{leminc1}}\label{sec7}

We have to show that 
$$\|D^{\alpha}(B\chi_{\Omega})\|_{L^p(\Omega)}\lesssim \|N\|_{\dot {B}_{p,p}^{\alpha-1/p}(\partial\Omega)},$$
where
$$D^{\alpha}(B\chi_{\Omega})(x)
=
\left(\int_{\Omega}\frac{|B\chi_{\Omega}(x)-B\chi_{\Omega}(y)|^{2}}{|x-y|^{2+2\alpha}}\,dm(y)\right)^{1/2}.$$
We will assume that $\Omega$ is a (bounded) Lipschitz domain.

Consider a decomposition of $\Omega$ into a family $\WW(\Omega)$ of Whitney squares.
We note that
\begin{align}\label{eqi1i2}
\|D^{\alpha}B\chi_{\Omega}\|_{p}^{p}
&=\sum_{Q\in \WW(\Omega)}\int_{Q}
\biggl(\sum_{R\in \WW(\Omega)}\int_{R}\frac{|B\chi_{\Omega}(x)-B\chi_{\Omega}(y)|^{2}}{|x-y|^{2+2\alpha}}\,dm(y)\biggr)^{p/2}dm(x)\\
&\lesssim\sum_{Q\in \WW(\Omega)}\int_{Q}
\biggl(\sum_{\substack{R\in\WW(\Omega):\\2R\cap2Q\neq\varnothing}}\int_{R}\frac{|B\chi_{\Omega}(x)-B\chi_{\Omega}(y)|^{2}}{|x-y|^{2+2\alpha}}\,dm(y)\biggr)^{p/2}dm(x)\nonumber\\
&\quad +\sum_{Q\in \WW(\Omega)}\int_{Q}\biggl(\sum_{\substack{R\in\WW(\Omega):\\2R\cap2Q=\varnothing}}
\int_{R}\frac{|B\chi_{\Omega}(x)-B\chi_{\Omega}(y)|^{2}}{|x-y|^{2+2\alpha}}\,dm(y)\biggr)^{p/2}dm(x) \nonumber\\
&=:I_{1}+I_{2}.\nonumber
\end{align}

\subsection{Estimate of $I_1$} \label{subsec7.1}
 From the properties of the Whitney decomposition, we know
that $\frac12\, Q\leq R\leq 2\ell(Q)$ for the squares $Q$ and $R$ involved in $I_1$. 
It follows easily that then $R\subset 8Q$. 
Let $L_Q$ be a line  that minimizes $\beta_1(Q)$ and
let $\Pi_Q$ be a half plane such that $\partial\Pi_Q=L_Q$ which contains $R$ and $Q$ 
(assuming $\beta_1(c_5\phi(Q))$ small enough, for some constant $c_5>1$).
From Lemma \ref{lemmaPi2}, we know that 
$B\chi_{\Pi_{Q}}(x)-B\chi_{\Pi_{Q}}(y)= 0$ for $x\in Q$ and $y\in R$.
Then we have
\begin{align}\label{eq:diff}
|B\chi_{\Omega}(x)-B\chi_{\Omega}(y)|&=|B\chi_{\Omega}(x)-B\chi_{\Omega}(y)-B\chi_{\Pi_{Q}}(x)+B\chi_{\Pi_{Q}}(y)|\\
&\leq\int_{\Omega\Delta\Pi_Q}\left|\frac{1}{(z-x)^{2}}-\frac{1}{(z-y)^{2}}\right|\,dm(z)\nonumber \\
&\lesssim\int_{\Omega\Delta\Pi_{Q}}\frac{|x-y|}{|z-x|^{3}}\,dm(z).\nonumber
\end{align}
Recall that, by the estimate \rf{nucl}, 
\begin{equation}
\int_{\Omega\Delta\Pi_{Q}}\frac{1}{|z-x|^{3}}\,dm(z)\lesssim\sum_{P\in\mathcal{D}(\partial\Omega):P\supset\phi(Q)}\frac{\beta_{1}(P)}{\ell(P)}+\frac{1}{\mathrm{diam}(\Omega)},\label{eq:re2}\end{equation}
and so we have
\begin{equation}\label{eqfo9}
|B\chi_{\Omega}(x)-B\chi_{\Omega}(y)|\lesssim |x-y|\,
\biggl(\,\sum_{P\in\mathcal{D}(\partial\Omega):P\supset\phi(Q)}\frac{\beta_{1}(P)}{\ell(P)}+\frac{1}{\mathrm{diam}(\Omega)}\biggr).
\end{equation}
It is easy to check that the preceding inequality also holds if $\beta_1(c_5\phi(Q))$ is not
assumed to be small.

Then, from \rf{eq:diff} and (\ref{eq:re2}) we deduce
\begin{align*}
I_{1} & \leq  \sum_{Q\in \WW(\Omega)}\int_{Q}\left(\int_{8Q}\frac{|B\chi_{\Omega}(x)-B\chi_{\Omega}(y)|^{2}}{|x-y|^{2+2\alpha}}dm(y)\right)^{p/2}\,dm(x)\\
 & \lesssim  \sum_{Q\in \WW(\Omega)}\int_{Q}\biggl(\int_{8Q}\frac{|x-y|^{2}}{|x-y|^{2+2\alpha}}\biggl(\,\sum_{P\in\mathcal{D}(\partial\Omega):P\supset\phi(Q)}\!\!\frac{\beta_{1}(P)}{\ell(P)}+\frac{1}{\mathrm{diam}(\Omega)}\biggr)^{2}\!dm(y)\biggr)^{p/2}dm(x)\\
 & \lesssim  \sum_{Q\in \WW(\Omega)}\int_{Q}\biggl(\ell(Q)^{2-2\alpha}\biggl(\,\sum_{P\in\mathcal{D}(\partial\Omega):P\supset\phi(Q)}\frac{\beta_{1}(P)}{\ell(P)}+\frac{1}{\mathrm{diam}(\Omega)}\biggr)^{2}\biggr)^{p/2}dm(x)\\
 & \lesssim  \sum_{Q\in \WW(\Omega)}\ell(Q)^{2+p-\alpha p}\biggl(\,\sum_{P\in\mathcal{D}(\partial\Omega):P\supset\phi(Q)}\frac{\beta_{1}(P)}{\ell(P)}+\frac{1}{\mathrm{diam}(\Omega)}\biggr)^{p}.
  \end{align*}

By the Cauchy-Schwarz inequality, it follows easily that, for any arbitrary $\ve>0$,
\begin{equation}\label{eqig55}
\biggl(\,\sum_{P\in\mathcal{D}(\partial\Omega):P\supset\phi(Q)}\frac{\beta_{1}(P)}{\ell(P)}
\biggr)^p\leq c \sum_{P\in\mathcal{D}(\partial\Omega):P\supset\phi(Q)}\frac{\beta_{1}(P)^{p}}{\ell(P)^{p-\ve}}\frac{1}{\ell(Q)^{\ve}},
\end{equation}
with $c$ depending on $\ve$.
Thus we get
  \begin{align}\label{eqig56}
 I_1
 & \lesssim  \sum_{Q\in \WW(\Omega)}\ell(Q)^{2+p-\alpha p-\ve}\!\!\!\sum_{P\in\mathcal{D}(\partial\Omega):P\supset\phi(Q)}\frac{\beta_{1}(P)^{p}}{\ell(P)^{p-\ve}}+\sum_{Q\in \WW(\Omega)}\frac{\ell(Q)^{2+p-\alpha p}}{\mathrm{diam}(\Omega)^p} \\
 & = \sum_{P\in \DD(\partial\Omega)}\frac{\beta_{1}(P)^{p}}{\ell(P)^{p-\ve}}\sum_{\phi(Q)
 \in\DD(\partial\Omega):\phi(Q)\subset P}\ell(Q)^{2+p-\alpha p-\ve}+\sum_{Q\in \WW(\Omega)}\frac{\ell(Q)^{2+p-\alpha p}}{\mathrm{diam}(\Omega)^p}. \nonumber
 \end{align}
Choosing $\ve$ small enough, we will have
 $2+p-\alpha p-\ve>1,$
which implies that 
$$\sum_{\phi(Q)
 \in\DD(\partial\Omega):\phi(Q)\subset P}\ell(Q)^{2+p-\alpha p-\ve}\lesssim
\ell(P)^{2+p-\alpha p-\ve}.$$
Analogously, we have
$\sum_{Q\in \WW(\Omega)}\ell(Q)^{2+p-\alpha p}\leq\diam(\Omega)^{2+p-\alpha p}$, since
$2+p-\alpha p>1$.
Therefore, we  obtain
$$I_1\lesssim \sum_{P\in \DD(\partial\Omega)}\beta_1(P)^{p}\,\ell(P)^{2-\alpha p}+\diam
(\Omega)^{2-\alpha p}.$$

\subsection{Estimate of $I_2$} \label{subsec7.2}
Now we deal with the term $I_{2}$ in \rf{eqi1i2}. Let $Q,R\in\WW(\Omega)$ satisfy
$2Q\cap2R=\varnothing$. Let $S_{Q,R}\in\DD(\partial\Omega)$ be such that $\phi(Q)\subset
S_{Q,R}$ and $\ell(S_{Q,R})\approx D(Q,R)$. 
Given $x\in Q$ and $y\in R$, let $z_{x,y}\in\Omega$ be the center of $\psi(S_{Q,R})$
(recall that $\psi(S_{Q,R})\in\WW(\Omega)$ was defined at the end of Subsection \ref{sec2.1}). 
Observe also that for $x\in Q$ and $y\in R$ with $2Q\cap2R=\varnothing$, we have
$|x-y|\approx D(Q,R)$.
We split $I_2$
as follows:
\begin{align}\label{eqdefiii}
I_{2} & \lesssim  
\sum_{Q\in \WW(\Omega)}\int_{Q}\biggl(\sum_{\substack{R\in\WW(\Omega):\\2R\cap2Q=\varnothing}}\int_{R}\frac{|B\chi_{\Omega}(x)-B\chi_{\Omega}(z_{x,y})|^{2}}{D(Q,R)^{2+2\alpha}}\,dm(y)\biggr)^{p/2}dm(x) \\
&\quad + \sum_{Q\in \WW(\Omega)}\int_{Q}\biggl(\sum_{\substack{R\in\WW(\Omega):\\2R\cap2Q=\varnothing}}\int_{R}\frac{|B\chi_{\Omega}(z_{x,y})-B\chi_{\Omega}(y)|^{2}}{D(Q,R)^{2+2\alpha}}\,dm(y)\biggr)^{p/2}dm(x)\nonumber\\
 & =  I_{2,1}+I_{2,2}.\nonumber
 \end{align}
 
\subsubsection{\bf Estimate of $I_{2,1}$.} 
Take cubes $Q_i\in\DD(\partial\Omega)$, $i=0,\ldots,m$, with $\ell(Q_i)=2^i\,\ell(\phi(Q))$
 such that
 $$\phi(Q)=Q_0\subset Q_1\subset Q_2\subset \ldots \subset Q_m =S_{Q,R}.$$
For $i=1,\ldots,m$, let $x_i$ be the center of $\psi(Q_i)$, and
set $x_0=x$ too. Then,
$$
|B\chi_{\Omega}(x)-B\chi_{\Omega}(z_{x,y})|  \leq  \sum_{i=0}^{m-1}|B\chi_{\Omega}(x_{i})-B\chi_{\Omega}(x_{i+1})|.$$
An estimate analogous to \rf{eqfo9} also holds, replacing $x$ by $x_i$ and $y$ by $x_{i+1}$.
Then we get
\begin{equation}\label{eqfo4}
|B\chi_{\Omega}(x)-B\chi_{\Omega}(z_{x,y})|  \lesssim
\sum_{\substack{P\in\DD(\partial\Omega):\\ \phi(Q)\subset P\subset S_{Q,R}}}\!\!\ell(P)\,\biggl(\,\sum_{
T\in\mathcal{D}(\partial\Omega):\,T\supset P}\frac{\beta_{1}(T)}{\ell(T)}+\frac{1}{\mathrm{diam}(\Omega)}\biggr).
\end{equation}
Now, recalling that $\ell(S_{Q,R})\approx D(Q,R)$, we obtain
\begin{align*}
\sum_{\substack{P\in\DD(\partial\Omega):\\ \phi(Q)\subset P\subset S_{Q,R}}}\!\!\ell(P)\sum_{
\substack{T\in\mathcal{D}(\partial\Omega):\\T\supset P}}\frac{\beta_{1}(T)}{\ell(T)}
& =
\sum_{\substack{T\in\DD(\partial\Omega):\\ T\supset\phi(Q)}}
\frac{\beta_{1}(T)}{\ell(T)}\,\sum_{\substack{P\in\DD(\partial\Omega):\\ \phi(Q)\subset P\subset 
S_{Q,R}\cap T}}
\!\ell(P)\\ &
\approx \sum_{\substack{T\in\DD(\partial\Omega):\\ T\supset\phi(Q)}}
\frac{\beta_{1}(T)}{\ell(T)}\,\min\bigl(\ell(T),D(Q,R)\bigr).
\end{align*}
Using the Cauchy-Schwarz inequality, it 
follows easily that, for any arbitrary $\ve$ with $0<\ve<p$,
$$\sum_{\substack{T\in\DD(\partial\Omega):\\ T\supset\phi(Q)}}
\frac{\beta_{1}(T)}{\ell(T)}\,\min\bigl(\ell(T),D(Q,R)\bigr)\leq \biggl(\,\,
\sum_{\substack{T\in\DD(\partial\Omega):\\ T\supset\phi(Q)}}
\frac{\beta_{1}(T)^p}{\ell(T)^\ve}\,D(Q,R)^\ve\biggr)^{1/p}.$$
The details are left for the reader. 

From \rf{eqfo4} and the last estimate we derive
\begin{equation}\label{eqfo93}
|B\chi_{\Omega}(x)-B\chi_{\Omega}(z_{x,y})|\lesssim D(Q,R)^{\ve/p}
\biggl(\,
\sum_{\substack{T\in\DD(\partial\Omega):\\ T\supset\phi(Q)}}
\frac{\beta_{1}(T)^p}{\ell(T)^\ve}\biggr)^{1/p} + \frac{D(Q,R)}{\diam(\Omega)}.
\end{equation}
Thus, for $x\in Q$, we have
\begin{align}\label{eqfo92}
\sum_{\substack{R\in\WW(\Omega):\\2R\cap2Q=\varnothing}} & \int_{R}
\frac{|B\chi_{\Omega}(x)-B\chi_{\Omega}(z_{x,y})|^{2}}{D(Q,R)^{2+2\alpha}}
\,dm(y) \\
&\lesssim 
\sum_{\substack{R\in\WW(\Omega):\\2R\cap2Q=\varnothing}}  
\frac{\ell(R)^2}{D(Q,R)^{2+2\alpha-2\ve/p}}
\biggl(\,
\sum_{\substack{T\in\DD(\partial\Omega):\\ T\supset\phi(Q)}}
\frac{\beta_{1}(T)^p}{\ell(T)^\ve}\biggr)^{2/p}\nonumber
\\
&\quad + \sum_{\substack{R\in\WW(\Omega):\\2R\cap2Q=\varnothing}}  
\frac{\ell(R)^2}{\diam(\Omega)^2\,D(Q,R)^{2\alpha}}.\nonumber
\end{align}
Concerning the first summand on the right side, notice that if $\ve$ is chosen small enough so that 
\begin{equation}\label{eqcond1e}
\alpha\,p-\ve >0,
\end{equation}
then
\begin{align*}
\sum_{\substack{R\in\WW(\Omega):\\2R\cap2Q=\varnothing}} & 
\frac{\ell(R)^2}{D(Q,R)^{2+2\alpha-2\ve/p}}
\biggl(\,
\sum_{\substack{T\in\DD(\partial\Omega):\\ T\supset\phi(Q)}} 
\frac{\beta_{1}(T)^p}{\ell(T)^\ve}\biggr)^{2/p}\\
& \lesssim \biggl(\,\sum_{\substack{T\in\DD(\partial\Omega):\\ T\supset\phi(Q)}}
\frac{\beta_{1}(T)^p}{\ell(T)^\ve}\biggr)^{2/p}
\int \frac1{\bigl(\ell(Q)+|x-y|\bigr)^{2+2\alpha-2\ve/p}}\,dm(y)\\
&\lesssim \biggl(\,\sum_{\substack{T\in\DD(\partial\Omega):\\ T\supset\phi(Q)}}
\frac{\beta_{1}(T)^p}{\ell(T)^\ve}\biggr)^{2/p}\,\frac1{\ell(Q)^{2\alpha-2\ve/p}}
\end{align*}

For the last summand in \rf{eqfo92}, using that $\diam(\Omega)\gtrsim D(Q,R)$, 
we get
\begin{align*}
 \sum_{\substack{R\in\WW(\Omega):\\2R\cap2Q=\varnothing}}  
\frac{\ell(R)^2}{\diam(\Omega)^2\,D(Q,R)^{2\alpha}}& \lesssim
\sum_{\substack{R\in\WW(\Omega)}}  
\frac{\ell(R)^2}{\diam(\Omega)^{2\alpha-1/p}\,D(Q,R)^{2+1/p}}\\
&\lesssim
 \sum_{P\in\DD(\partial\Omega)}  
\frac{\ell(P)^2}{\diam(\Omega)^{2\alpha-1/p}\,D(\phi(Q),P)^{2+1/p}}.
\end{align*}
Then, from Lemma \ref{lemmaCor} we deduce that
\begin{equation}\label{eqd33''}
 \sum_{\substack{R\in\WW(\Omega):\\2R\cap2Q=\varnothing}}  
\frac{\ell(R)^2}{\diam(\Omega)^2\,D(Q,R)^{2\alpha}}\lesssim 
\frac1{\diam(\Omega)^{2\alpha-1/p}\,\ell(Q)^{1/p}}.
\end{equation}

Therefore, we have
\begin{align*}
\sum_{\substack{R\in\WW(\Omega):\\2R\cap2Q=\varnothing}}  \int_{R} 
\frac{|B\chi_{\Omega}(x)-B\chi_{\Omega}(z_{x,y})|^{2}}{D(Q,R)^{2+2\alpha}}
\mathrm{d}y & \lesssim 
\biggl(\,\sum_{\substack{T\in\DD(\partial\Omega):\\ T\supset\phi(Q)}}
\frac{\beta_{1}(T)^p}{\ell(T)^\ve}\biggr)^{2/p}\,\frac1{\ell(Q)^{2\alpha-2\ve/p}}\\
&\quad
+ \frac1{\diam(\Omega)^{2\alpha-1/p}\,\ell(Q)^{1/p}}.
\end{align*}
Recalling the definition of $I_{2,1}$ in \rf{eqdefiii}, we get
\begin{align}\label{eqdf98}
I_{2,1}  &\lesssim  
\sum_{Q\in \WW(\Omega)}\int_{Q}\Biggl[
\biggl(\,\sum_{\substack{T\in\DD(\partial\Omega):\\ T\supset\phi(Q)}}
\frac{\beta_{1}(T)^p}{\ell(T)^\ve}\biggr)^{2/p}\,\frac1{\ell(Q)^{2\alpha-2\ve/p}}
+ \frac1{\diam(\Omega)^{2\alpha-1/p}\,\ell(Q)^{1/p}}\Biggr]^{p/2}\!dm(x)\\
&\lesssim
\sum_{Q\in \WW(\Omega)}
\biggl(\,\sum_{\substack{T\in\DD(\partial\Omega):\\ T\supset\phi(Q)}}
\frac{\beta_{1}(T)^p}{\ell(T)^\ve}\biggr)\,\frac{\ell(Q)^2}{\ell(Q)^{\alpha p-\ve}}
+ \sum_{Q\in \WW(\Omega)}\frac{\ell(Q)^2}{\diam(\Omega)^{\alpha p-1/2}\,\ell(Q)^{1/2}}\nonumber\\
& =
\sum_{Q\in \WW(\Omega)}\,\,
\sum_{\substack{T\in\DD(\partial\Omega):\\ T\supset\phi(Q)}}
\frac{\beta_{1}(T)^p}{\ell(T)^\ve}\,\ell(Q)^{2-\alpha p+\ve}
+ \sum_{Q\in \WW(\Omega)}\frac{\ell(Q)^{3/2}}{\diam(\Omega)^{\alpha p-1/2}}.\nonumber
\end{align}
Suppose now that
\begin{equation}\label{eqcond2e}
2-\alpha p+\ve>1.
\end{equation}
Then we get
\begin{align*}
\sum_{Q\in \WW(\Omega)}\,\,
\sum_{\substack{T\in\DD(\partial\Omega):\\ T\supset\phi(Q)}}
\frac{\beta_{1}(T)^p}{\ell(T)^\ve}\,\ell(Q)^{2-\alpha p+\ve}&
\approx 
\sum_{T\in\DD(\partial\Omega)}\frac{\beta_{1}(T)^p}{\ell(T)^\ve}
\sum_{P\in\DD(\partial\Omega):\,P\subset T}
\ell(P)^{2-\alpha p+\ve/2} \\
&\approx \sum_{T\in\DD(\partial\Omega)}\frac{\beta_{1}(T)^p}{\ell(T)^\ve}
\ell(T)^{2-\alpha p+\ve} = \sum_{T\in\DD(\partial\Omega)}\beta_{1}(T)^p
\ell(T)^{2-\alpha p}
\end{align*}
For the last summand on the right side of \rf{eqdf98} we use that
$$\sum_{Q\in \WW(\Omega)}\ell(Q)^{3/2}\approx \sum_{P\in \DD(\partial\Omega)}\ell(P)^{3/2}
\lesssim\diam(\Omega)^{3/2}.$$
So we finally obtain
$$I_{2,1}\lesssim \sum_{T\in\DD(\partial\Omega)}\beta_{1}(T)^p 
\ell(T)^{2-\alpha p} + \diam(\Omega)^{2-\alpha\,p}.$$

Notice now that if we choose $\ve=\alpha p - \alpha/2$, say, then $0<\ve<p$ and both 
\rf{eqcond1e} and \rf{eqcond2e} hold.

\subsubsection{\bf Estimate of $I_{2,2}$}
We argue as we did for $I_{2,1}$. We take $R_i\in\DD(\partial\Omega)$, $i=0,\ldots,m'$, with $\ell(R_i)=2^i\,\ell(\phi(R))$
 such that
 $$\phi(R)=R_0\subset R_1\subset R_2\subset\ldots\subset R_{m'} =:S_{R,Q},$$
 where $S_{R,Q}\in\DD(\partial\Omega)$ satisfies $\ell(S_{R,Q})\approx D(Q,R)$. 
Notice that 
$$\dist(S_{Q,R},S_{R,Q})\lesssim D(Q,R)\approx \ell(S_{Q,R}) \approx \ell(S_{R,Q}).$$
Thus $z_{x,y}$ (the center of $S_{Q,R}$) belongs to $c\,S_{R,Q}$, for some fixed constant $c>1$, and $\dist(z_{x,y},\partial\Omega)
\approx \ell(S_{R,Q})$.
Then, as in the case of $I_{2,1}$ in \rf{eqfo93}, for any $0<\ve<p$ (to be fixed later), we get
\begin{equation}\label{eqfo93'}
|B\chi_{\Omega}(y)-B\chi_{\Omega}(z_{x,y})|\lesssim D(Q,R)^{\ve/p}
\biggl(\,
\sum_{\substack{T\in\DD(\partial\Omega):\\ T\supset\phi(R)}}
\frac{\beta_{1}(T)^p}{\ell(T)^\ve}\biggr)^{1/p} + \frac{D(Q,R)}{\diam(\Omega)},
\end{equation}
and
for $x\in Q$, we have
\begin{align}\label{eqfo92'}
\sum_{\substack{R\in\WW(\Omega):\\2R\cap2Q=\varnothing}} & \int_{R}
\frac{|B\chi_{\Omega}(x)-B\chi_{\Omega}(z_{x,y})|^{2}}{D(Q,R)^{2+2\alpha}}
\,dm(y) \\
&\lesssim 
\sum_{\substack{R\in\WW(\Omega):\\2R\cap2Q=\varnothing}}  
\frac{\ell(R)^2}{D(Q,R)^{2+2\alpha-2\ve/p}}
\biggl(\,
\sum_{\substack{T\in\DD(\partial\Omega):\\ T\supset\phi(R)}}
\frac{\beta_{1}(T)^p}{\ell(T)^\ve}\biggr)^{2/p}\nonumber
\\
&\quad + \sum_{\substack{R\in\WW(\Omega):\\2R\cap2Q=\varnothing}}  
\frac{\ell(R)^2}{\diam(\Omega)^2\,D(Q,R)^{2\alpha}}.\nonumber
\end{align}

To simplify notation, denote
$$\alpha(R)=\sum_{\substack{T\in\DD(\partial\Omega):\\ T\supset\phi(R)}}
\frac{\beta_{1}(T)^p}{\ell(T)^\ve}.$$
Applying the Cauchy-Schwarz inequality to the first term on the right side of \rf{eqfo92'},
we obtain
\begin{multline}\label{eqdf83}
\biggl(\sum_{R\in\WW(\Omega)}  
\frac{\ell(R)^2}{D(Q,R)^{2+2\alpha-2\ve/p}} \,\alpha(R)^{2/p}\biggr)^{p/2}\\
\leq  \biggl(\sum_{R\in\WW(\Omega)}\alpha(R)\frac{\ell(R)^{a}}{D(Q,R)^{b}}\biggr)\biggl(\sum_{R\in\WW(\Omega)}\frac{\ell(R)^{1+\delta}}{D(Q,R)^{1+2\delta}}\biggr)^{\frac{p}{2}-1}.
\end{multline}
where $\delta>0$ will be chosen below and
\begin{align}\label{eqaaa}
a&=p-(1+\delta)\left(\frac{p}{2}-1\right),\\
b&=p+\alpha p-\ve-(1+2\delta)\left(\frac{p}{2}-1\right).\label{eqbbb}
\end{align}
By Lemma \ref{lemmaCor}, the last sum in \rf{eqdf83} is bounded by $c(\delta)/\ell(Q)^\delta$.
So we have
$$\biggl(\sum_{R\in\WW(\Omega)}  
\frac{\ell(R)^2}{D(Q,R)^{2+2\alpha-2\ve/p}} \,\alpha(R)^{2/p}\biggr)^{p/2}
\lesssim  \frac1{\ell(Q)^{\delta(\frac p2-1)}}\sum_{R\in\WW(\Omega)}\alpha(R)\frac{\ell(R)^{a}}{D(Q,R)^{b}}.$$

From \rf{eqfo92'}, the last estimate, \rf{eqd33''}, and the definition of $I_{2,2}$ in \rf{eqdefiii}, we get
\begin{align}\label{eqdf98'}
I_{2,2}  &\lesssim  
\sum_{Q\in \WW(\Omega)}\int_{Q}\;\biggl[
\frac1{\ell(Q)^{\delta(\frac p2-1)}}\sum_{R\in\WW(\Omega)}\alpha(R)\,\frac{\ell(R)^{a}}{D(Q,R)^{b}}
+ \frac1{\diam(\Omega)^{\alpha p-1/2}\,\ell(Q)^{1/2}}\biggr]\,dm(x)\\
&=
\sum_{Q\in \WW(\Omega)} \ell(Q)^{2-\delta\left(\frac{p}{2}-1\right)}\sum_{R\in\WW(\Omega)}\alpha(R)\,\frac{\ell(R)^{a}}{D(Q,R)^{b}}
+ \sum_{Q\in \WW(\Omega)}\frac{\ell(Q)^{3/2}}{\diam(\Omega)^{\alpha p-1/2}}.\nonumber
\end{align}
As in the case of $I_{1,2}$, to estimate the last sum we use that
$\sum_{Q\in \WW(\Omega)}\ell(Q)^{3/2}
\lesssim\diam(\Omega)^{3/2},$
and thus,
$$\sum_{Q\in \WW(\Omega)}\frac{\ell(Q)^{3/2}}{\diam(\Omega)^{p\alpha-1/2}}
\lesssim \diam(\Omega)^{2-\alpha\,p}.$$

Now we consider the first sum on the right side of \rf{eqdf98'}. This equals
\begin{align*}
J:=\sum_{Q\in \WW(\Omega)} \ell(Q)^{2-\delta\left(\frac{p}{2}-1\right)}&\sum_{R\in\WW(\Omega)}\,\,\sum_{\substack{T\in\DD(\partial\Omega):\\ T\supset\phi(R)}}
\frac{\beta_{1}(T)^p}{\ell(T)^\ve}\frac{\ell(R)^{a}}{D(Q,R)^{b}}\\
&\lesssim \sum_{T\in\DD(\partial\Omega)}\frac{\beta_{1}(T)^{p}}{\ell(T)^{\ve}}
\sum_{\substack{R\in\DD(\partial\Omega):\\R\subset T}}\ell(R)^{a}
\sum_{Q\in\DD(\partial\Omega)}\frac{\ell(Q)^{2-\delta\left(\frac{p}{2}-1\right)}}{D(Q,R)^{b}}.
\end{align*}
Assuming that
\begin{equation}\label{eqcondb}
b>2-\delta\left(\frac{p}{2}-1\right)>1,
\end{equation}
by Lemma \ref{lemmaCor} we obtain
\begin{align*} 
 J & \lesssim \sum_{T\in\DD(\partial\Omega)}\frac{\beta_{1}(T)^{p}}{\ell(T)^{\ve}}\sum_{\substack{R\in\DD(\partial\Omega):\\R\subset T}}\ell(R)^{a}\frac{1}{\ell(R)^{b-2+\delta\left(\frac{p}{2}-1\right)}}\\
& = \sum_{T\in\DD(\partial\Omega)}\frac{\beta_{1}(T)^{p}}{\ell(T)^{\ve}}\sum_{\substack{R\in\DD(\partial\Omega):\\R\subset T}}
 \ell(R)^{a-b+2-\delta\left(\frac{p}{2}-1\right)}
 \end{align*}
 Assuming also that
\begin{equation}\label{eqconda}
 a-b+2-\delta\left(\frac{p}{2}-1\right)>1,
\end{equation}
we get $\sum_{\substack{R\in\DD(\partial\Omega):R\subset T}}
 \ell(R)^{a-b+2-\delta\left(\frac{p}{2}-1\right)}\lesssim \ell(T)^{a-b+2-\delta\left(\frac{p}{2}-1\right)}$, and then,
$$J\lesssim \sum_{T\in\DD(\partial\Omega)}\frac{\beta_{1}(T)^{p}}{\ell(T)^{\ve}}\,
 \ell(T)^{a-b+2-\delta\left(\frac{p}{2}-1\right)} = \sum_{T\in\DD(\partial\Omega)}\beta_{1}(T)^p 
\ell(T)^{2-\alpha p}.$$

So finally we have
$$I_{2,2}\lesssim \sum_{T\in\DD(\partial\Omega)}\beta_{1}(T)^p 
\ell(T)^{2-\alpha p} + \diam(\Omega)^{2-\alpha\,p}.$$

Now it remains to check that the constants $\ve$ and $\delta$ can be chosen so that $0<\ve<p$,
$\delta>0$, and moreover
\rf{eqcondb} and \rf{eqconda} hold. We assume that $\delta>0$ is very small ($0<\delta\ll1$). 
Notice that the condition \rf{eqcondb} is equivalent to
$$\frac p2 + \alpha p +1 -\ve +O(\delta) >2+O(\delta)>1,$$
where, as usual, $O(\delta)$ stands for some term $\leq c\,\delta$, with $c$ possibly depending on $p$.
Since we need also $\ve<p$, the condition above suggests the choice
$$\ve=\min\left(p,\frac{p}{2}+\alpha p-1\right)-c_6\delta,$$
for some constant $c_6$ big enough (depending on $p$).
Let us see that indeed this is a good choice. It is clear that \rf{eqcondb} holds by construction,
and that $\ve<p$. 
 On the other hand, $\ve>0$ is equivalent to
\begin{equation}\label{eqcond52}
\frac{p}{2}+\alpha p-1>c_6\delta,
\end {equation}
which holds for $\delta$ small enough, under the assumption $\alpha p>1$ from the lemma.

Now we only have to check that \rf{eqconda} is also satisfied. By plugging the
values of $a$ and $b$, this is equivalent to
$$-\alpha p + \ve +1 >O(\delta).$$
This holds both if $\ve=p-c_6\delta$ (recall that $0<\alpha<1$), and also if
$\ve=\frac{p}{2}+\alpha p-1-c_6\delta$.

\subsection{The end of the proof}\label{subsecend}

From the estimates obtained for $I_1$, $I_{2,1}$ and $I_{2,2}$,
we deduce that
\begin{align*}
\|D^{\alpha}B\chi_{\Omega}\|_{L^p(\Omega)}^{p}&\lesssim \sum_{Q\in\DD(\partial\Omega)}\beta_{1}(Q)^p 
\ell(Q)^{2-\alpha p} + \diam(\Omega)^{2-\alpha\,p}\\
&= \sum_{Q\in\DD(\partial\Omega)}\biggl(\frac{\beta_{1}(Q)}{\ell(Q)^{\alpha-1/p}}\biggr)^p \ell(Q)
+ \diam(\Omega)^{2-\alpha\,p}\lesssim \|N\|_{\dot {B}_{p,p}^{\alpha-1/p}(\partial\Omega)}^p,
\end{align*}
by Lemma \ref{lemdorron} and the subsequent remark.

\subsection{The proof for special Lipschitz domains}
The arguments are very similar (and in fact, simpler) to the ones above for Lipschitz domains. The main difference stems
from the fact that the estimate \rf{nucl} holds without the summand 
$c/{\diam(\Omega)}$ on the right side.
As a consequence, all the terms above which involve $\diam(\Omega)$ do not appear in the case of
special Lipschitz domains.

\vv

\section{Proof of Lemma \ref{leminc2}}\label{sec8}

We have to show that 
\begin{equation}\label{eqda12}
\iint_{\Omega^2}\frac{|B\chi_{\Omega}(x)-B\chi_{\Omega}(y)|^{p}}{|x-y|^{2+\alpha p}}dm(x)
\,dm(y)\lesssim \|N\|_{\dot {B}_{p,p}^{\alpha-1/p}(\partial\Omega)}^p.
\end{equation}
First we will assume that $\Omega$ is a (bounded) Lipschitz domain.
The argument will be very similar, and even simpler, to the one in the preceding section for Lemma \ref{leminc1}.

Again  we consider a decomposition of $\Omega$ into a family $\WW(\Omega)$ of Whitney squares.
The integral above can be written as follows:
\begin{align}\label{eqi11}
\sum_{Q\in \WW(\Omega)} \sum_{R\in \WW(\Omega)} & \int_{Q}
\int_{R}\frac{|B\chi_{\Omega}(x)-B\chi_{\Omega}(y)|^{p}}{|x-y|^{2+\alpha p}}dm(x)
\, dm(y)\\
& =\sum_{Q\in \WW(\Omega)}
\sum_{\substack{R\in\WW(\Omega):\\2R\cap2Q\neq\varnothing}}\int_{Q}\int_{R}
\frac{|B\chi_{\Omega}(x)-B\chi_{\Omega}(y)|^{p}}{|x-y|^{2+\alpha p}}dm(x)
\, dm(y)\nonumber\\
&\quad +
\sum_{Q\in \WW(\Omega)}
\sum_{\substack{R\in\WW(\Omega):\\2R\cap2Q=\varnothing}}\int_{Q}\int_{R}
\frac{|B\chi_{\Omega}(x)-B\chi_{\Omega}(y)|^{p}}{|x-y|^{2+\alpha p}}dm(x) \nonumber
\, dm(y)\\
&=:I_{1}+I_{2}.\nonumber
\end{align}


\subsection{Estimate of $I_1$} 
As in Subsection \ref{subsec7.1}, now we have
  $R\subset 8Q$. 
From \rf{eqfo9} and \rf{eqig55}, for $x\in Q$ and $y\in R$ we infer that
$$
|B\chi_{\Omega}(x)-B\chi_{\Omega}(y)|^p\lesssim |x-y|^p\,\Biggl(
 \sum_{P\in\mathcal{D}(\partial\Omega):P\supset\phi(Q)}\frac{\beta_{1}(P)^{p}}{\ell(P)^{p-\ve}}\frac{1}{\ell(Q)^{\ve}}
+\frac{1}{\mathrm{diam}(\Omega)^p}\Biggr),
$$
for $\ve>0$. Thus we get
\begin{align*}
I_1\lesssim \!\sum_{Q\in\WW(\Omega)} \iint_{\begin{subarray}{l} x\in Q\\ y\in 8Q\end{subarray}}
\frac1{|x-y|^{2+\alpha p - p}}\,\Biggl(
 \sum_{\substack{P\in\mathcal{D}(\partial\Omega):\\P\supset\phi(Q)}}\frac{\beta_{1}(P)^{p}}{\ell(P)^{p-\ve}}\frac{1}{\ell(Q)^{\ve}}
+\frac{1}{\mathrm{diam}(\Omega)^p}\Biggr)dm(x)dm(y).
\end{align*}
Since $2+\alpha p - p<2$, we derive
$$I_1\lesssim \sum_{Q\in\WW(\Omega)} \Biggl(
 \sum_{\substack{P\in\mathcal{D}(\partial\Omega):\\P\supset\phi(Q)}}\frac{\beta_{1}(P)^{p}}{\ell(P)^{p-\ve}}\frac{1}{\ell(Q)^{\ve}}
+\frac{1}{\mathrm{diam}(\Omega)^p}\Biggr)\,\ell(Q)^{2-\alpha p +p}.$$
This is the same we got in \rf{eqig56}. So, as before, we obtain
 $$I_1\lesssim \sum_{P\in \DD(\partial\Omega)}\beta_1(P)^{p}\,\ell(P)^{2-\alpha p}+\diam
(\Omega)^{2-\alpha p}.$$


\subsection{Estimate of $I_2$}
Let $Q,R\in\WW(\Omega)$ be such that $2Q\cap2R=\varnothing$.
Given $x\in Q$ and $y\in R$, we define $z_{x,y}$ as in Subsection \ref{subsec7.2}.
From \rf{eqfo93} and \rf{eqfo93'} we deduce that
\begin{align*}
|B \chi_{\Omega}(x)&-B\chi_{\Omega}(y)|^p \lesssim |B\chi_{\Omega}(x)-B\chi_{\Omega}(z_{x,y})|^p +
|B\chi_{\Omega}(y)-B\chi_{\Omega}(z_{x,y})|^p\\
&\lesssim D(Q,R)^{\ve}
\Biggl(\,
\sum_{\substack{S\in\DD(\partial\Omega):\\ S\supset\phi(Q)}}
\frac{\beta_{1}(S)^p}{\ell(S)^\ve} + 
\sum_{\substack{T\in\DD(\partial\Omega):\\ T\supset\phi(R)}}
\frac{\beta_{1}(T)^p}{\ell(T)^\ve}\Biggr)
+
\frac{D(Q,R)^p}{\diam(\Omega)^p},
\end{align*} 
for $0<\ve<p$ to be fixed below.
Using also that $|x-y|\approx D(Q,R)$, we get
\begin{align*}
I_2\lesssim 
 \sum_{Q\in \WW(\Omega)} &
\sum_{R\in\WW(\Omega)}\int_{Q} \int_{R}
\frac1{D(Q,R)^{2+\alpha p}}\\
&\!\!\!\!\!\!\!\Biggl[
D(Q,R)^{\ve}
\Biggl(\,
\sum_{\substack{S\in\DD(\partial\Omega):\\ S\supset\phi(Q)}}
\frac{\beta_{1}(S)^p}{\ell(S)^\ve} + 
\sum_{\substack{T\in\DD(\partial\Omega):\\ T\supset\phi(R)}}
\frac{\beta_{1}(T)^p}{\ell(T)^\ve}\Biggr)
+
\frac{D(Q,R)^p}{\diam(\Omega)^p}
\Biggr]
dm(x)
\, dm(y).
\end{align*}
Then, because of the symmetry on $Q$ and $R$, 
\begin{align*}
I_2 &\lesssim 
 \sum_{Q\in \WW(\Omega)}
\sum_{R\in\WW(\Omega)}\int_{Q} \int_{R} \biggl[
\frac1{D(Q,R)^{2+\alpha p-\ve}}
\sum_{\substack{S\in\DD(\partial\Omega):\\ S\supset\phi(Q)}}\!\!
\frac{\beta_{1}(S)^p}{\ell(S)^\ve} \\
&\quad
+
\frac{1}{\diam(\Omega)^p\,D(Q,R)^{2+\alpha p-p}}\biggr]
dm(x) \, dm(y)\\
& =  
\sum_{Q\in \WW(\Omega)} \sum_{\substack{S\in\DD(\partial\Omega):\\ S\supset\phi(Q)}}
\frac{\beta_{1}(S)^p}{\ell(S)^\ve} \,\ell(Q)^2
\sum_{R\in\WW(\Omega)}\frac{\ell(R)^2}{D(Q,R)^{2+\alpha p-\ve}}\\
&\quad+
\sum_{Q\in \WW(\Omega)} 
\sum_{R\in\WW(\Omega)} \frac{\ell(Q)^2\,\ell(R)^2}{\diam(\Omega)^p\,D(Q,R)^{2+\alpha p-p}}. 
\end{align*}
The terms on the right side are estimate following the ideas used for $I_{2,1}$ in Subsection
\ref{subsec7.2}.
By Lemma \ref{lemmaCor}, we have
$$\sum_{R\in\WW(\Omega)}\frac{\ell(R)^2}{D(Q,R)^{2+\alpha p-\ve}}\lesssim \frac1{\ell(Q)^{\alpha p
-\ve}},$$
assuming 
$$\alpha p -\ve >0.$$
Therefore,
\begin{align*}
\sum_{Q\in \WW(\Omega)}  \sum_{\substack{S\in\DD(\partial\Omega):\\ S\supset\phi(Q)}}\!
\frac{\beta_{1}(S)^p}{\ell(S)^\ve} \,\ell(Q)^2\!\! &
\sum_{R\in\WW(\Omega)}\frac{\ell(R)^2}{D(Q,R)^{2+\alpha p-\ve}} 
\lesssim  \!\!\sum_{S\in\DD(\partial\Omega)}\! \frac{\beta_{1}(S)^p}{\ell(S)^\ve}\sum_{\substack{Q\in \WW(\Omega):\\
\phi(Q)\subset S}}
 \!\ell(Q)^{2-\alpha p+\ve} \\
& \approx \sum_{S\in\DD(\partial\Omega)} \frac{\beta_{1}(S)^p}{\ell(S)^\ve} \,\ell(S)^{2-\alpha p+\ve}
= \sum_{S\in\DD(\partial\Omega)} \beta_{1}(S)^p\,\ell(S)^{2-\alpha p},
\end{align*}
assuming also that 
$$2-\alpha p+\ve>1$$
in the second estimate.
Finally, arguing as in the case of $I_{1,2}$ (see \rf{eqd33''}), we also get
$$\sum_{Q\in \WW(\Omega)} 
\sum_{R\in\WW(\Omega)} \frac{\ell(Q)^2\,\ell(R)^2}{\diam(\Omega)^p\,D(Q,R)^{2+\alpha p-p}}\lesssim
\diam
(\Omega)^{2-\alpha p}.$$

If we choose $\ve = \alpha p -\alpha /2$, then all the above conditions involving $\ve$ are satisfied, 
and so we have
$$I_2\lesssim
\sum_{Q\in \DD(\partial\Omega)}\beta_1(Q)^{p}\,\ell(Q)^{2-\alpha p}+\diam
(\Omega)^{2-\alpha p}.$$
Together with the estimates obtained for $I_1$, using Lemma \ref{lemdorron} and the subsequent remark,
this yields \rf{eqda12}.

\vv

For special Lipschitz domains, the arguments are very similar to the ones above. 
The difference stems
from the fact that the estimate \rf{nucl} holds without the summand 
$c/{\diam(\Omega)}$ on the right side, and thus
 all the terms above which involve $\diam(\Omega)$ do not appear in the case of
special Lipschitz domains.


\section{The case $\alpha\, p\leq 1$ and a final remark}\label{secfi}

\subsection{The case $\alpha\, p\leq 1$}
In this situation, the estimate
\begin{equation}\label{eqd444}
\sum_{Q\in\DD(\partial\Omega)}\biggl(\frac{\beta_{1}(Q)}{\ell(Q)^{\alpha-1/p}}\biggr)^p \ell(Q)
\lesssim \|N\|_{\dot {B}_{p,p}^{\alpha-1/p}(\partial\Omega)}^p
\end{equation}
no longer holds, since for the application of Lemmas \ref{lemanorm} and \ref{lemdorron} one
needs $0<\alpha-1/p<1$. However, if $\Omega$ is a special Lipschitz domain, by Dorronsoro's theorem we still have
$$\sum_{Q\in\DD(\partial\Omega)}\biggl(\frac{\beta_{1}(Q)}{\ell(Q)^{\alpha-1/p}}\biggr)^p \ell(Q)
\lesssim \|A\|_{\dot {B}_{p,p}^{1+\alpha-1/p}}^p,$$
where $A:\R\to\R$ is the Lipschitz function that parameterizes $\partial\Omega$. In the case 
$\Omega$ is a bounded Lipschitz domain, then the sum above can be estimate also in terms of the
$\dot {B}_{p,p}^{1+\alpha-1/p}(\R)$ norms of the local parameterizations of $\partial
\Omega$.

In the proof of Theorem \ref{teo}, apart from the estimate \rf{eqd444}, all the other arguments
work for $p=1$. Then one obtains:

\begin{theorem}\label{teo'}
Let $\Omega\subset\C$ be a Lipschitz domain such that in each ball $B(z,R)$, 
with $z\in\partial\Omega$, $\partial\Omega\cap B(z,R)$ coincides with the graph of a Lipschitz
function $A_z:\R\to\R$ such that $A_z\in \dot {B}_{1,1}^1(\R)$, then
$\partial B(\chi_\Omega)\in L^1(\Omega)$.

If $\Omega$ is a special Lipschitz domain, so that $\Omega=\{(x,y)\in\C:\,
y>A(x)\}$, where $A:\R\to\R$ is a Lipschitz function with $\|A'\|_\infty\leq\delta$ . Then we have
$$\|\partial B(\chi_\Omega)\|_{L^1(\Omega)}\leq c \,\|A\|_{\dot {B}_{1,1}^1},$$
with $c$ depending on $\delta$.
\end{theorem}
 
Analogously, for $0<\alpha<1$ and $1\leq p<\infty$ with $\alpha\,p\leq 1$, all the arguments in
the proof of Lemmas \ref{leminc1} and \ref{leminc2} work with the exception of \rf{eqd444}, under 
the additional assumption that
\begin{equation}\label{eqas64}
\alpha\,p +\frac p2>1
\end{equation}
in the case of Lemma \ref{leminc1} (this is used in \rf{eqcond52}). 

Recalling also that
\begin{equation}\label{eqd447}
\|B(\chi_\Omega)\|_{\dot W^{\alpha,p}(\Omega)}\lesssim \|B(\chi_\Omega)\|_{\dot B^\alpha_{p,p}(\Omega)}
\qquad \mbox{if $1< p\leq2$,}
\end{equation}
it turns out that, to estimate $\|B(\chi_\Omega)\|_{\dot W^{\alpha,p}(\Omega)}$ we can apply Lemma
\ref{leminc2} for $1<p\leq2$, and use Lemma \ref{leminc1} for the case $p\geq2$, so that the assumption \rf{eqas64} is fulfilled.

To summarize, we have:

\begin{theorem}\label{teo2'}
Let $\Omega\subset\C$ be either a Lipschitz or 
a special Lipschitz domain, 
and let $1\leq p<\infty$ and $0<\alpha<1$. Suppose that the Lipschitz functions $A_z$ 
which give the local parameterization of $\partial \Omega$ in case $\Omega$ is bounded (defined as in Theorem \ref{teo'}), or the function $y=A(x)$ if $\Omega$ is a special Lipschitz domain (with
$A$ compactly supported), 
belong to $\dot {B}_{p,p}^{1+\alpha-1/p}(\R)$.
Then,
\begin{itemize}
\item If $p>1$ or, in the case $p=1$,  $\alpha>\dfrac 12$, then $B(\chi_\Omega)\in \dot W^{\alpha,p}(\Omega)$.  Moreover, if $\Omega$ is a special Lipschitz domain, then
$$\|B(\chi_\Omega)\|_{\dot W^{\alpha,p}(\Omega)}\lesssim
\|A\|_{\dot {B}_{p,p}^{1+
\alpha-1/p}}.$$

\item For $1\leq p <\infty$,  $B(\chi_\Omega)\in \dot B^\alpha_{p,p}(\Omega)$.
 Moreover, if $\Omega$ is a special Lipschitz domain, then
$$\|B(\chi_\Omega)\|_{\dot B^\alpha_{p,p}(\Omega)}\lesssim
\|A\|_{\dot {B}_{p,p}^{1+
\alpha-1/p}}.$$

\end{itemize}
\end{theorem}

\vv

Finally, notice that if $\alpha\,p<1$ and $A$ is Lipschitz with compact support, then
$\|A\|_{\dot B_{p,p}^{1+\alpha-1/p}}<\infty$, since $1+\alpha-1/p<1$. As a consequence,
$$B(\chi_\Omega)\in \dot W^{\alpha,p}(\Omega) \quad\mbox{if $\alpha\,p +\frac p2>1$,}$$
and 
$$B(\chi_\Omega)\in \dot B^\alpha_{p,p}(\Omega).$$ 
From the last two statements and \rf{eqd447}, we infer that
$$B(\chi_\Omega)\in \dot W^{\alpha,p}(\Omega) \quad\mbox{if $p>1$.}$$ 
So we have:

\begin{theorem}\label{teo3'}
Let $\Omega\subset\C$ be either a Lipschitz or 
a special Lipschitz domain. Let $1<p<\infty$ and $0<\alpha<1$ be such that $\alpha\,p<1$.
Then,  $B(\chi_\Omega)\in \dot W^{\alpha,p}(\Omega) \cap \dot B^\alpha_{p,p}(\Omega)$.
\end{theorem}

\vv
\subsection{A final remark}

The techniques and results in this paper can be extended easily to the case of even homogeneous 
Calder\'on-Zygmund
operators in $\R^n$. Indeed, if $T:L^p(\R^n)\to L^p(\R^n)$ is such an operator, then
for any ball $B\subset \R^n$, 
$$T\chi_B(x) = 0\qquad \mbox{for $x\in B$.}$$
See Lemma 3 from \cite{Mateu-Orobitg-Verdera}.
From this result, it turns out that $\nabla T\chi_B = 0$ on $B$ and also, for any half hyperplane $\Pi\subset
\R^n$, 
$$\nabla T\chi_\Pi=0 \qquad\mbox{for $x\not\in \partial \Pi$.}$$
Then one can argue as in the proof of Theorems \ref{teo} and \ref{teo2} and obtain analogous results for $T$.


\end{document}